\documentclass[11pt,a4paper]{amsart}
\usepackage{mathrsfs}
\usepackage{syntonly}
\usepackage{amsmath}
\usepackage{amsthm}
\usepackage{amsfonts}
\usepackage{amssymb}
\usepackage{latexsym}
\usepackage{amscd,amssymb,amsopn,amsmath,amsthm,graphics,amsfonts,mathrsfs,accents,enumerate,verbatim,calc}
\usepackage[dvips]{graphicx}
\usepackage[colorlinks=true,linkcolor=blue,citecolor=blue]{hyperref}
\usepackage[all]{xy}

\date{}
\allowdisplaybreaks[4] \footskip=15pt
\renewcommand{\uppercasenonmath}[1]{}

\numberwithin{equation}{section} \theoremstyle{plain}
\newtheorem{lem}{Lemma}[section]
\newtheorem{cor}[lem]{Corollary}
\newtheorem{prop}[lem]{Proposition}
\newtheorem{thm}[lem]{Theorem}

\newtheorem{cond}[lem]{Condition}
\newtheorem{definition}[lem]{Definition}
\newtheorem{Ex}[lem]{Example}
\newtheorem{Quest}[lem]{Question}
\newtheorem{Property}[lem]{Property}
\newtheorem{Properties}[lem]{Properties}
\newtheorem{Subprops}{}[lem]
\newtheorem{Para}[lem]{}

\newtheorem{rem}[lem]{Remark}

\newtheorem*{ack*}{ACKNOWLEDGEMENTS}

\newcommand{\pf}{\noindent\begin {proof}}
\newcommand{\epf}{\end{proof}}

\newcommand{\C}{\mathcal{C}}

\newcommand{\h}{{\rm H}}

\begin{document}
\begin{center}
{\large  \bf  Proper classes and Gorensteinness in extriangulated categories}

\vspace{0.5cm}  Jiangsheng Hu$^{a}$, Dongdong Zhang$^{b}$\footnote{Corresponding author. Jiangsheng Hu was supported by the NSF of China (Grant No. 11501257, 11671069, 11771212),  Qing Lan Project of Jiangsu Province and Jiangsu Government Scholarship for Overseas Studies (JS-2019-328).   Panyue Zhou was supported by the National Natural Science Foundation of China (Grant No. 11901190 and 11671221), the Hunan Provincial Natural Science Foundation of China (Grant No. 2018JJ3205) and the Scientific Research Fund of Hunan Provincial Education Department (Grant No. 19B239).} and Panyue Zhou$^{c}$ \\
\medskip

\hspace{-4mm}$^{a}$School of Mathematics and Physics, Jiangsu University of Technology
 Changzhou 213001, China\\
 $^b$Department of Mathematics, Zhejiang Normal University,
\small Jinhua 321004, China\\
$^c$College of Mathematics, Hunan Institute of Science and Technology, Yueyang, Hunan 414006, China\\
E-mails: jiangshenghu@jsut.edu.cn, zdd@zjnu.cn and panyuezhou@163.com \\
\end{center}

\bigskip
\centerline { \bf  Abstract}
\medskip

\leftskip10truemm \rightskip10truemm \noindent
 \hspace{1em}Extriangulated categories were introduced by Nakaoka and Palu as a simultaneous generalization of exact categories and triangulated categories. A notion of proper class in an extriangulated category is defined in this paper. Let $\mathcal{C}$ be an extriangulated category and $\xi$ a proper class in $\mathcal{C}$. We prove that $\mathcal{C}$ admits a new extriangulated structure. This construction gives
 extriangulated categories which are neither exact categories nor triangulated categories. Moreover, we introduce and study $\xi$-Gorenstein projective objects in $\mathcal{C}$ and demonstrate that $\xi$-Gorenstein projective objects  share some basic properties with Gorenstein projective objects in module categories or in triangulated categories. In particular, we refine a result of Asadollahi and Salarian [Gorenstein objects in
triangulated categories, J. Algebra 281(2004), 264-286]. As an application, the admissible model structure on extriangulated categories is obtained.\\[2mm]
{\bf Keywords:} Extriangulated category; proper class; Gorenstein projective object; model structure.\\
{\bf 2010 Mathematics Subject Classification:} 18E30; 18E10; 16E05; 18G20; 18G35.

\leftskip0truemm \rightskip0truemm
\section { \bf Introduction}


Relative homological algebra has been formulated by Hochschild in
categories of modules and later by Heller and Butler and Horrocks in more general categories with a relative abelian structure. Beligiannis developed in \cite{Bel1} a relative version of homological algebra in triangulated categories in analogy to relative homological algebra in abelian categories, in which the notion of a proper class of exact sequences is replaced by a proper class of triangles. By specifying a class of triangles $\mathcal{E}$, which is called a proper class of triangles, he introduced $\mathcal{E}$-projective and $\mathcal{E}$-injective objects.

In general, it is not so easy to find a proper class $\mathcal{E}$ of triangles in a triangulated category having enough $\mathcal{E}$-projectives or $\mathcal{E}$-injectives. One of interesting examples, due to Krause \cite{Krause} and Beligiannis \cite{Bel1}, is as follows:

\begin{itemize} \item Assume that $(\mathcal{T},\Sigma, \Delta)$ is a compactly generated triangulated category, where $\Sigma$ is the suspension functor and $\Delta$ is the triangulation. Then the class $\mathcal{E}$ of pure triangles
(which is induced by the compact objects) is proper and $\mathcal{T}$ has enough $\mathcal{E}$-projectives or $\mathcal{E}$-injectives.
\end{itemize}

\noindent Note that $\mathcal{E}$ is a proper class in the above triangulated category $(\mathcal{T},\Sigma, \Delta)$. It follows that  the triangle $$\Sigma^iA\xrightarrow{~(-1)^i\Sigma^if~}\Sigma^iB\xrightarrow{~(-1)^i\Sigma^ig~}\Sigma^iC\xrightarrow{~(-1)^i\Sigma^ih~}\Sigma^{i+1} A$$ belongs to $\mathcal{E}$ for all $i\in \mathbb{Z}$ provided that $\xymatrix@C=1.5em{A\ar[r]^f&B\ar[r]^g&C\ar[r]^{h\ \ \ \ }&\Sigma A}$ belongs to $\mathcal{E}$. However,  one can get that $(\mathcal{T},\Sigma, \mathcal{E})$ is not a triangulated category because not every morphism $f\colon X\rightarrow Y$ in $\mathcal{T}$ can be embeded into the pure triangle $$X\xrightarrow{~f~}Y\rightarrow Z\rightarrow \Sigma X.$$
This means that there exist a triangulated category $(\mathcal{T},\Sigma, \Delta)$ and a proper class $\mathcal{E}$ in $\mathcal{T}$ such that
$(\mathcal{T},\Sigma, \mathcal{E})$ is not triangulated.

Recently, the notion of an extriangulated category was introduced in \cite{NP}, which is a simultaneous generalization of
exact category and triangulated category. Of course, there are extriangulated
categories which are neither exact categories nor triangulated categories, see \cite[Proposition 3.30]{NP} and \cite[Example 4.14]{ZZ}.

We know that extriangulated categories generalize both triangulated and exact categories.
Motivated by this, for any extriangulated categoriy $\mathcal{C}$, we wonder whether or not there exists a proper class $\xi$ of $\mathbb{E}$-triangles such that $\mathcal{C}$ is equipped with the proper class $\xi$ of $\mathbb{E}$-triangles, which is again extriangulated. This is the first aim of this paper.

Auslander and Bridger \cite{AB} introduced modules of G-dimension zero for finitely generated modules over a commutative Noetherian ring as a generalization of finitely generated projective modules. Enochs and Jenda \cite{EJ1} introduced Gorenstein projective modules that generalize the notion of G-dimension zero to any module over any ring. Dually they defined Gorenstein injective modules, and they developed a relative homological algebra in the category of modules. Furthermore, Beligiannis \cite{Bel2} introduced the concept of an $\mathcal{X}$-Gorenstein object { in an additive category $\mathcal{C}$} for a contravariantly finite subcategory $\mathcal{X}$ of $\mathcal{C}$ such that any $\mathcal{X}$-epic has kernel in $\mathcal{C}$ as a generalization of modules of G-dimension zero in the sense of Auslander and Bridger.
In an attempt to develop the Beligiannis' theory, Asadollahi and Salarian \cite{AS1} introduced and studied $\xi$-Gorenstein projective and injective objects in triangulated categories { with a proper class $\xi$}, which share basic properties with Gorenstein projective and injective modules.
Our next aim of this paper is to contribute in developing the above mentioned homological algebra in extriangulated categories.

We now outline the results of the paper. In Section 2, we summarize some preliminaries
and basic facts about extriangulated categories which will be used throughout the paper.

In Section 3, for a given extriangulated category $(\mathcal{C}, \mathbb{E}, \mathfrak{s})$, we define a notion of a proper class of $\mathbb{E}$-triangles, denoted by $\xi$.
If an extriangulated category $\mathcal{C}$ is equipped with a proper class of $\mathbb{E}$-triangles $\xi$, we show that $\mathcal{C}$ admits a new extriangulated structure (see Theorem \ref{thma}). This construction gives extriangulated categories which are neither exact categories nor triangulated categories (see Remark \ref{rem:3.4}).

In Section 4, let $\xi$ be a proper class of $\mathbb{E}$-triangles in an extriangulated category $(\mathcal{C}, \mathbb{E}, \mathfrak{s})$. We introduce
$\xi$-projective objects, $\xi$-$\mathcal{G}$projective objects and their duals. Denote by $\mathcal{P}(\xi)$ the class of $\xi$-projective objects and by $\mathcal{GP}(\xi)$ the class of $\xi$-$\mathcal{G}$projective objects in $\mathcal{C}$. We prove that the category $\mathcal{GP}(\xi)$ is full, additive, closed under isomorphisms, direct summands provided that $(\mathcal{C}, \mathbb{E}, \mathfrak{s})$ satisfies Condition (WIC) (see Theorem \ref{thm2} and Condition \ref{cond:4.11}). As a corollary, we refine a result of Asadollahi and Salarian in \cite{AS1} (see Remark \ref{rem:4.19}).

In Section 5, by using the class of $\xi$-$\mathcal{G}$projective objects in an extriangulated category $(\mathcal{C}, \mathbb{E}, \mathfrak{s})$, we relate an invariant called $\xi$-$\mathcal{G}$projective dimension ($\xi$-$\mathcal{G}$${\rm pd} A$ for short) to any object $A$ of $\mathcal{C}$ and show that this invariant has some nice properties (see Propositions \ref{pro6} and \ref{thm3}). As an application, the admissible model structure on extriangulated categories is obtained, which { generalizes} a result of Yang \cite[Theorem 5.7]{Yang} (see Theorem \ref{thm:5.4} and Corollary \ref{yang}).

\section{\bf Preliminaries}
Let us briefly recall some definitions and basic properties of extriangulated categories from \cite{NP}. Throughout this paper, we assume that $\mathcal{C}$ is an additive category.

 \begin{definition}\cite[Definition 2.1]{NP} {\rm Suppose that $\mathcal{C}$ is equipped with an additive bifunctor $$\mathbb{E}: \mathcal{C}^{op}\times \mathcal{C}\rightarrow {\rm Ab},$$  where ${\rm Ab}$ is the category of abelian groups. For any objects $A, C\in\mathcal{C}$, an element $\delta\in \mathbb{E}(C,A)$ is called an $\mathbb{E}$-extension. Thus formally, an $\mathbb{E}$-extension is a triple $(A,\delta, C)$. For any $A, C\in\mathcal{C}$, the zero element $0\in\mathbb{E}(C, A)$ is called the split $\mathbb{E}$-extension.

 Let $\delta\in \mathbb{E}(C, A)$ be any $\mathbb{E}$-extension. By the functoriality, for any $a\in\mathcal{C}(A, A')$ and $c\in\mathcal{C}(C', C)$, we have $\mathbb{E}$-extensions
 \begin{center}$\mathbb{E}(C, a)(\delta)\in\mathbb{E}(C, A')$ ~and~ $\mathbb{E}(c, A)(\delta)\in\mathbb{E}(C', A)$.\end{center}
 We abbreviately denote them by $a_*\delta$ and $c^*\delta$ { respectively}. In this terminology, we have $$\mathbb{E}(c, a)(\delta)=c^*a_*\delta=a_*c^*\delta$$ in $\mathbb{E}(C', A')$.
{ In the following, we always assume that the  bifunctor $\mathbb{E}$ exists. }}
 \end{definition}

 \begin{definition}\cite[Definition 2.3]{NP} {\rm Let $\delta\in\mathbb{E}(C, A)$ and $\delta'\in\mathbb{E}(C', A')$ be any pair of $\mathbb{E}$-extensions. A morphism $(a, c): \delta\rightarrow \delta'$ of $\mathbb{E}$-extensions is a pair of morphisms $a\in\mathcal{C}(A, A')$ and $c\in\mathcal{C}(C, C')$ in $\mathcal{C}$ satisfying the equality $$a_*\delta=c^*\delta'.$$}
 \end{definition}

 \begin{definition}\cite[Definition 2.6]{NP} {\rm Let $\delta=(A, \delta, C)$ and $\delta'=(A', \delta', C')$ be any pair of $\mathbb{E}$-extensions. Let
 \begin{center}$\xymatrix{C\ar[r]^{\iota_C\ \ \ }&C\oplus C'&C'\ar[l]_{\qquad\iota_{C'}}}$
 and
 $\xymatrix{A&A\oplus A'\ar[l]_{ p_A\;\;}\ar[r]^{\quad p_{A'} }&A'}$\end{center}
be coproduct and product in $\mathcal{C}$, respectively. Remark that, by the additivity of $\mathbb{E}$, we have a natural isomorphism
\begin{center} $\mathbb{E}(C\oplus C', A\oplus A')\simeq\mathbb{E}(C, A)\oplus\mathbb{E}(C, A')
\oplus\mathbb{E}(C', A)\oplus\mathbb{E}(C', A')$.\end{center}
Let $\delta\oplus \delta'\in\mathbb{E}(C\oplus C', A\oplus A')$ be the element corresponding to $(\delta, 0, 0, \delta')$ through this isomorphism. This is the unique element which satisfies
\begin{center}$\mathbb{E}(\iota_C, p_A)(\delta\oplus \delta')=\delta,\ \mathbb{E}(\iota_C, p_{A'})(\delta\oplus \delta')=0, \ \mathbb{E}(\iota_{C'}, p_A)(\delta\oplus \delta')=0,\ \mathbb{E}(\iota_{C'}, p_{A'})(\delta\oplus \delta')=\delta'$.\end{center}}
 \end{definition}

 \begin{definition}\cite[Definition 2.7]{NP} {\rm Let $A, C\in\mathcal{C}$ be any pair of objects. Two sequences of morphisms in $\mathcal{C}$
 \begin{center} $\xymatrix{A\ar[r]^x&B\ar[r]^y&C}$ and $\xymatrix{A\ar[r]^{x'}&B'\ar[r]^{y'}&C}$\end{center}
 are said to be equivalent if there exists an isomorphism $b\in\mathcal{C}(B, B')$ which makes the following diagram commutative.
 $$\small\xymatrix@C=1.2em@R=0.5cm{&&B\ar[dd]^b_\simeq\ar[drr]^y&&\\
 A\ar[urr]^x\ar[drr]_{x'}&&&&C\\
 &&B'\ar[urr]_{y'}}$$

 We denote the equivalence class of $\xymatrix{A\ar[r]^x&B\ar[r]^y&C}$ by $\xymatrix{[A\ar[r]^x&B\ar[r]^y&C].}$}
 \end{definition}

 \begin{definition}\cite[Definition 2.8]{NP} {\rm
 {(1)} For any $A, C\in\mathcal{C}$, we denote as
 $$0=[A\xrightarrow{~\tiny\begin{bmatrix}1\\0\end{bmatrix}~}A\oplus C\xrightarrow{\tiny\begin{bmatrix}0&1\end{bmatrix}}C].$$

{(2)} For any $\xymatrix{[A\ar[r]^x&B\ar[r]^y&C]}$ and $\xymatrix{[A'\ar[r]^{x'}&B'\ar[r]^{y'}&C']}$, we denote as
\begin{center} $\xymatrix{[A\ar[r]^x&B\ar[r]^y&C]}\oplus$$\xymatrix{[A'\ar[r]^{x'}&B'\ar[r]^{y'}&C']}=$$\xymatrix{[A\oplus A'\ar[r]^{x\oplus x'}&B\oplus B'\ar[r]^{y\oplus y'}&C\oplus C'].}$\end{center}}
 \end{definition}

 \begin{definition}\label{def1}\cite[Definition 2.9]{NP} {\rm
  Let $\mathfrak{s}$ be a correspondence which associates an equivalence class $$\mathfrak{s}(\delta)=\xymatrix@C=0.8cm{[A\ar[r]^x
 &B\ar[r]^y&C]}$$ to any $\mathbb{E}$-extension $\delta\in\mathbb{E}(C, A)$. This $\mathfrak{s}$ is called a {\it realization} of $\mathbb{E}$, if it satisfies
 the following condition $(\star)$. In this case, we say that the sequence $\xymatrix{A\ar[r]^x&B\ar[r]^y&C}$ realizes $\delta$, whenever it satisfies $\mathfrak{s}(\delta)=$$\xymatrix{[A\ar[r]^x &B\ar[r]^y&C].}$

 $(\star)$ Let $\delta\in\mathbb{E}(C, A)$ and $\delta'\in\mathbb{E}(C', A')$ be any pair of $\mathbb{E}$-extensions, with
 \begin{center} $\mathfrak{s}(\delta)=$$\xymatrix{[A\ar[r]^x &B\ar[r]^y&C]}$ and $\mathfrak{s}(\delta')=$$\xymatrix{[A'\ar[r]^{x'} &B'\ar[r]^{y'}&C'].}$\end{center} Then, for any morphism $(a, c): \delta\rightarrow \delta'$, there exists $b\in\mathcal{C}(B, B')$ which makes the following diagram commutative
 $$\xymatrix{A\ar[r]^{x}\ar[d]_{a}&B\ar[r]^{y}\ar[d]_{b}&C\ar[d]_{c}\\
 A'\ar[r]^{x'}&B'\ar[r]^{y'}&C'.}$$
 In the above situation, we say that the triplet $(a, b, c)$ realizes $(a, c)$.}
 \end{definition}

 \begin{definition} \cite[Definition 2.10]{NP} {\rm
 Let $\mathcal{C}, \mathbb{E}$ be as above. A realization of $\mathbb{E}$ is said to be {\it additive}, if it satisfies the following conditions.

{\rm (i)} For any $A, C\in\mathcal{C}$, the split $\mathbb{E}$-extension $0\in\mathbb{E}(C, A)$ satisfies $\mathfrak{s}(0)=0.$

{\rm (ii)} For any pair of $\mathbb{E}$-extensions $\delta\in\mathbb{E}(C, A)$ and $\delta'\in\mathbb{E}(C', A')$, we have \begin{center}$\mathfrak{s}(\delta\oplus\delta')=\mathfrak{s}(\delta)\oplus\mathfrak{s}(\delta')$.\end{center}}
 \end{definition}

 \begin{definition} \cite[Definition 2.12]{NP} {\rm A triplet $(\mathcal{C}, \mathbb{E}, \mathfrak{s})$ is called an {\it extriangulated category} if it satisfies the following conditions.

{\rm(ET1)} $\mathbb{E}: \mathcal{C}^{op}\times \mathcal{C}\rightarrow \rm{Ab}$ is an additive bifunctor.

{\rm (ET2)} $\mathfrak{s}$ is an additive realization of $\mathbb{E}$.

{\rm (ET3)} Let $\delta\in\mathbb{E}(C, A)$ and $\delta'\in\mathbb{E}(C', A')$ be any pair of $\mathbb{E}$-extensions, realized as
 \begin{center} $\mathfrak{s}(\delta)=$$\xymatrix{[A\ar[r]^x &B\ar[r]^y&C]}$ and $\mathfrak{s}(\delta')=$$\xymatrix{[A'\ar[r]^{x'} &B'\ar[r]^{y'}&C'].}$\end{center} For any commutative square
 $$\xymatrix{A\ar[r]^{x}\ar[d]_{a}&B\ar[r]^{y}\ar[d]_{b}&C\\
 A'\ar[r]^{x'}&B'\ar[r]^{y'}&C'}$$
 in $\mathcal{C}$, there exists a  morphism $(a, c): \delta\rightarrow \delta'$ satisfying $cy=y'b$.

{\rm (ET3)}$^{\rm op}$ Dual of {\rm (ET3)}.

{\rm (ET4)} Let $\delta\in\mathbb{E}(D,A)$ and $\delta'\in\mathbb{E}(F, B)$ be $\mathbb{E}$-extensions realized by
 \begin{center} $\xymatrix{A\ar[r]^f&B\ar[r]^{f'}&D}$ and $\xymatrix{B\ar[r]^g&C\ar[r]^{g'}&F}$\end{center}
 respectively. Then there exist an object $E\in\mathcal{C}$, a commutative diagram
 $$\xymatrix{A\ar[r]^f\ar@{=}[d]&B\ar[r]^{f'}\ar[d]_g&D\ar[d]^d\\
A\ar[r]^h&C\ar[r]^{h'}\ar[d]_{g'}&E\ar[d]^e\\
&F\ar@{=}[r]&F
}$$
in $\mathcal{C}$, and an $\mathbb{E}$-extension $\delta^{''}\in\mathbb{E}(E, A)$ realized by $\xymatrix{A\ar[r]^h&C\ar[r]^{h'}&E,}$
which satisfy the following compatibilities.

{\rm (i)} $\xymatrix{D\ar[r]^d&E\ar[r]^{e}&F}$ realizes $f'_*\delta'$,

{\rm (ii)} $d^*\delta^{''}=\delta$,

{\rm (iii)} $f_*\delta^{''}=e^*\delta'$.

{\rm (ET4)}$^{\rm op}$ Dual of {\rm (ET4)}.
} \end{definition}

We will use the following terminology.

\begin{definition}{ \cite[Definitions 2.15 and 2.19]{NP}} {\rm
 Let $(\mathcal{C}, \mathbb{E}, \mathfrak{s})$ be an extriangulated category.

{\rm (1)} A sequence $\xymatrix@C=1cm{A\ar[r]^x&B\ar[r]^{y}&C}$ is called a {\it conflation} if it realizes some $\mathbb{E}$-extension $\delta\in\mathbb{E}(C, A)$.
In this case, $x$ is called an {\it inflation} and $y$ is called a {\it deflation}.

{\rm (2)} If a conflation $\xymatrix@C=0.6cm{A\ar[r]^x&B\ar[r]^{y}&C}$ realizes $\delta\in\mathbb{E}(C, A)$, we call the pair
$(\xymatrix@C=0.6cm{A\ar[r]^x&B\ar[r]^{y}&C}, \delta)$ an {\it $\mathbb{E}$-triangle}, and write it in the following way.
\begin{center} $\xymatrix{A\ar[r]^x&B\ar[r]^{y}&C\ar@{-->}[r]^{\delta}&}$\end{center}
We usually do not write this ``$\delta$" if it is not used in the argument.

{\rm (3)} Let $\xymatrix{A\ar[r]^x&B\ar[r]^{y}&C\ar@{-->}[r]^{\delta}&}$ and $\xymatrix{A'\ar[r]^{x'}&B'\ar[r]^{y'}&C'\ar@{-->}[r]^{\delta'}&}$
be any pair of $\mathbb{E}$-triangles. If a triplet $(a, b, c)$ realizes $(a, c): \delta\rightarrow \delta'$, then we write it as
 $$\xymatrix{A\ar[r]^{x}\ar[d]_{a}&B\ar[r]^{y}\ar[d]_{b}&C\ar[d]_{c}\ar@{-->}[r]^{\delta}&\\
 A'\ar[r]^{x'}&B'\ar[r]^{y'}&C'\ar@{-->}[r]^{\delta'}&}$$
 and call $(a, b, c)$ a {\it morphism} of $\mathbb{E}$-triangles.}
\end{definition}

\begin{lem}\label{lem1} \emph{(see \cite[Proposition 3.15]{NP})} Let $(\mathcal{C}, \mathbb{E},\mathfrak{s})$ be an extriangulated category. Then the following hold.

\emph{(1)} Let $C$ be any object, and let $\xymatrix@C=2em{A_1\ar[r]^{x_1}&B_1\ar[r]^{y_1}&C\ar@{-->}[r]^{\delta_1}&}$ and $\xymatrix@C=2em{A_2\ar[r]^{x_2}&B_2\ar[r]^{y_2}&C\ar@{-->}[r]^{\delta_2}&}$ be any pair of $\mathbb{E}$-triangles. Then there is a commutative diagram
in $\mathcal{C}$
$$\xymatrix{
    & A_2\ar[d]_{m_2} \ar@{=}[r] & A_2 \ar[d]^{x_2} \\
  A_1 \ar@{=}[d] \ar[r]^{m_1} & M \ar[d]_{e_2} \ar[r]^{e_1} & B_2\ar[d]^{y_2} \\
  A_1 \ar[r]^{x_1} & B_1\ar[r]^{y_1} & C   }
  $$
  which satisfies $\mathfrak{s}(y^*_2\delta_1)=\xymatrix@C=2em{[A_1\ar[r]^{m_1}&M\ar[r]^{e_1}&B_2]}$ and
  $\mathfrak{s}(y^*_1\delta_2)=\xymatrix@C=2em{[A_2\ar[r]^{m_2}&M\ar[r]^{e_2}&B_1].}$

  \emph{(2)} Let $A$ be any object, and let $\xymatrix@C=2em{A\ar[r]^{x_1}&B_1\ar[r]^{y_1}&C_1\ar@{-->}[r]^{\delta_1}&}$ and $\xymatrix@C=2em{A\ar[r]^{x_2}&B_2\ar[r]^{y_2}&C_2\ar@{-->}[r]^{\delta_2}&}$ be any pair of $\mathbb{E}$-triangles. Then there is a commutative diagram
in $\mathcal{C}$
$$\xymatrix{
     A\ar[d]_{x_2} \ar[r]^{x_1} & B_1 \ar[d]^{m_2}\ar[r]^{y_1}&C_1\ar@{=}[d] \\
  B_2 \ar[d]_{y_2} \ar[r]^{m_1} & M \ar[d]^{e_2} \ar[r]^{e_1} & C_1 \\
  C_2 \ar@{=}[r] & C_2 &   }
  $$
  which satisfies $\mathfrak{s}(x_{2*}\delta_1)=\xymatrix@C=2em{[B_2\ar[r]^{m_1}&M\ar[r]^{e_1}&C_1]}$ and $\mathfrak{s}(x_{1*}\delta_2)=\xymatrix@C=2em{[B_1\ar[r]^{m_2}&M\ar[r]^{e_2}&C_2].}$
\end{lem}
Assume that $(\mathcal{C}, \mathbb{E}, \mathfrak{s})$ is an extriangulated category. By Yoneda's Lemma, any $\mathbb{E}$-extension $\delta\in \mathbb{E}(C, A)$ induces  natural transformations
\begin{center} $\delta_\sharp: \mathcal{C}(-, C)\Rightarrow \mathbb{E}(-, A)$ and $\delta^\sharp: \mathcal{C}(A, -)\Rightarrow \mathbb{E}(C, -)$.\end{center}
For any $X\in\mathcal{C}$, these $(\delta_\sharp)_X$ and $\delta^\sharp_X$ are given as follows:

(1) $(\delta_\sharp)_X: \mathcal{C}(X, C)\Rightarrow \mathbb{E}(X, A); ~f\mapsto f^*\delta.$

(2) $\delta^\sharp_X: \mathcal{C}(A, X)\Rightarrow \mathbb{E}(C, X); ~g\mapsto g_*\delta.$
\begin{lem} {\rm \cite[Corollary 3.12]{NP}}  Let $(\mathcal{C}, \mathbb{E}, \mathfrak{s})$ be an extriangulated category, and $$\xymatrix@C=2em{A\ar[r]^{x}&B\ar[r]^{y}&C\ar@{-->}[r]^{\delta}&}$$ an $\mathbb{E}$-triangle. Then we have the following long exact sequences:

$\xymatrix@C=1cm{\mathcal{C}(C, -)\ar[r]^{\mathcal{C}(y, -)}&\mathcal{C}(B, -)\ar[r]^{\mathcal{C}(x, -)}&\mathcal{C}(A, -)\ar[r]^{\delta^\sharp}&\mathbb{E}(C, -)\ar[r]^{\mathbb{E}(y, -)}&\mathbb{E}(B, -)\ar[r]^{\mathbb{E}(x, -)}&\mathbb{E}(A, -);}$

$\xymatrix@C=1cm{\mathcal{C}(-, A)\ar[r]^{\mathcal{C}(-, x)}&\mathcal{C}(-, B)\ar[r]^{\mathcal{C}(-, y)}&\mathcal{C}(-, C)\ar[r]^{\delta_\sharp}&\mathbb{E}(-, A)\ar[r]^{\mathbb{E}(-, x)}&\mathbb{E}(-, B)\ar[r]^{\mathbb{E}(-, y)}&\mathbb{E}(-, C).}$

\end{lem}
\section{\bf Proper classes of $\mathbb{E}$-triangles}
  Throughout this section, $(\mathcal{C}, \mathbb{E}, \mathfrak{s})$ is  an extriangulated category.
A class of $\mathbb{E}$-triangles $\xi$ is {\it closed under base change} if for any $\mathbb{E}$-triangle $$\xymatrix@C=2em{A\ar[r]^x&B\ar[r]^y&C\ar@{-->}[r]^{\delta}&\in\xi}$$ and any morphism $c\colon C' \to C$, then any $\mathbb{E}$-triangle  $\xymatrix@C=2em{A\ar[r]^{x'}&B'\ar[r]^{y'}&C'\ar@{-->}[r]^{c^*\delta}&}$ belongs to $\xi$.

Dually, a class of  $\mathbb{E}$-triangles $\xi$ is {\it closed under cobase change} if for any $\mathbb{E}$-triangle $$\xymatrix@C=2em{A\ar[r]^x&B\ar[r]^y&C\ar@{-->}[r]^{\delta}&\in\xi}$$ and any morphism $a\colon A \to A'$, then any $\mathbb{E}$-triangle  $\xymatrix@C=2em{A'\ar[r]^{x'}&B'\ar[r]^{y'}&C\ar@{-->}[r]^{a_*\delta}&}$ belongs to $\xi$.

A class of $\mathbb{E}$-triangles $\xi$ is called {\it saturated} if in the situation of Lemma \ref{lem1}(1), whenever {
 $\xymatrix@C=2em{A_2\ar[r]^{x_2}&B_2\ar[r]^{y_2}&C\ar@{-->}[r]^{\delta_2 }&}$
 and $\xymatrix@C=2em{A_1\ar[r]^{m_1}&M\ar[r]^{e_1}&B_2\ar@{-->}[r]^{y_2^{\ast}\delta_1}&}$ }
 belong to $\xi$, then the  $\mathbb{E}$-triangle $\xymatrix@C=2em{A_1\ar[r]^{x_1}&B_1\ar[r]^{y_1}&C\ar@{-->}[r]^{\delta_1 }&}$  belongs to $\xi$.

An $\mathbb{E}$-triangle $\xymatrix@C=2em{A\ar[r]^x&B\ar[r]^y&C\ar@{-->}[r]^{\delta}&}$ is called {\it split} if $\delta=0$. It is easy to see that it is split if and only if $x$ is section or $y$ is retraction. The full subcategory  consisting of the split $\mathbb{E}$-triangles will be denoted by $\Delta_0$.

  \begin{definition} \label{def:proper class}{\rm  Let $\xi$ be a class of $\mathbb{E}$-triangles which is closed under isomorphisms. $\xi$ is called a {\it proper class} of $\mathbb{E}$-triangles if the following conditions hold:

  (1) $\xi$ is closed under finite coproducts and $\Delta_0\subseteq \xi$.

  (2) $\xi$ is closed under base change and cobase change.

  (3) $\xi$ is saturated.}

  \end{definition}

%
%
The following is the main result of this section.
\begin{thm}\label{thma} Let $\xi$ be a class of $\mathbb{E}$-triangles which is closed under isomorphisms.
Set $\mathbb{E}_\xi:=\mathbb{E}|_\xi$, that is, $$\mathbb{E}_\xi(C, A)=\{\delta\in\mathbb{E}(C, A)~|~\delta~ \textrm{is realized as an $\mathbb{E}$-triangle}\xymatrix{A\ar[r]^x&B\ar[r]^y&C\ar@{-->}[r]^{\delta}&}~\textrm{in}~\xi\}$$ for any $A, C\in\mathcal{C}$, and $\mathfrak{s}_\xi:=\mathfrak{s}|_{\mathbb{E}_\xi}$. Then $\xi$ is a  proper class  of $\mathbb{E}$-triangles if and only if $(\mathcal{C}, \mathbb{E}_\xi, \mathfrak{s}_\xi)$ is an extriangulated category.
\end{thm}

\begin{rem} \label{rem:3.4} \emph{(1)}  Assume that $(\mathcal{C}, \mathbb{E}, \mathfrak{s})$ is a compactly generated triangulated category and $\xi$ is the class of pure triangles
(which is induced by the compact objects).  It follows from Theorem \ref{thma} that $(\mathcal{C}, \mathbb{E}_\xi, \mathfrak{s}_\xi)$ is an extriangulated category. Since the inflation in $(\mathcal{C}, \mathbb{E}_\xi, \mathfrak{s}_\xi)$ is not a monomorphism in the
categorical sense, hence $(\mathcal{C}, \mathbb{E}_\xi, \mathfrak{s}_\xi)$  is not an exact category. On the other hand, $(\mathcal{C}, \mathbb{E}_\xi, \mathfrak{s}_\xi)$ is not a triangulated category because not every morphism $f\colon X\rightarrow  Y$ in $\mathcal{C}$ can be embeded into the pure triangle in
$\xi$. { This example comes from \cite{Krause}.}

\emph{(2)} Assume that $(\mathcal{C}, \mathbb{E}, \mathfrak{s})$ is an exact category and $\xi$ is a class of exact sequences which is closed under isomorphisms. One can check that $\xi$ is a proper class if and only if $(\mathcal{C}, \mathbb{E}_\xi, \mathfrak{s}_\xi)$ is an exact category.

\emph{(3)} If $\mathcal{C}$ is a triangulated category and the class $\xi$ of triangles is closed under
isomorphisms and suspension \emph{(see \cite[Section 2.2]{Bel1})}, then $\xi$ is a proper class if and only if $(\mathcal{C}, \mathbb{E}_\xi, \mathfrak{s}_\xi)$ is an extriangulated category. However $(\mathcal{C}, \mathbb{E}_\xi, \mathfrak{s}_\xi)$ is not a triangulated category by \emph{(1)} in general.
\end{rem}

\begin{definition} {\rm Let $\xi$ be a proper class of $\mathbb{E}$-triangles.
A morphism $x$ is called {\it $\xi$-inflation} if there exists an  $\mathbb{E}$-triangle $\xymatrix{A\ar[r]^x&B\ar[r]^{y}&C\ar@{-->}[r]^{\delta}&}$  in $\xi$.
A morphism $y$ is called {\it $\xi$-deflation} if there exists an  $\mathbb{E}$-triangle $\xymatrix{A\ar[r]^x&B\ar[r]^{y}&C\ar@{-->}[r]^{\delta}&}$  in $\xi$.}
\end{definition}

By Theorem \ref{thma} and  \cite[Remark 2.16]{NP}, we have the following corollary.
\begin{cor}\label{cor1}  Let $\xi$ be a proper class of $\mathbb{E}$-triangles.
 Then the class of $\xi$-inflations {\rm(}respectively $\xi$-deflations{\rm)} is closed under compositions.
\end{cor}

To prove Theorem \ref{thma}, we need the following lemmas.
\begin{lem}\label{lemb}
\emph{(1)} If $\xymatrix{A\ar[r]^{x}&B\ar[r]^{y}&C\ar@{-->}[r]^{\delta}&}$ and $\xymatrix{D\ar[r]^{d}&E\ar[r]^{e}&C\ar@{-->}[r]^{\rho}&}$
 are $\mathbb{E}$-triangles, and $g: B\rightarrow E$ is a morphism satisfying $eg=y$, then there exists a morphism $f: A\rightarrow D$ which gives a morphism of $\mathbb{E}$-triangles
$$\xymatrix@C=2em{A\ar[r]^x\ar@{-->}[d]^f&B\ar[r]^y\ar[d]^g&C\ar@{-->}[r]^{\delta}\ar@{=}[d]&\\
  D\ar[r]^d&E\ar[r]^e&C\ar@{-->}[r]^{\rho}&.
  }$$
  Moreover, $\xymatrix@C=1,2cm{A\ar[r]^{\tiny\begin{bmatrix}-f\\x\end{bmatrix}\ \ \ }&D\oplus B\ar[r]^{\tiny\ \ \begin{bmatrix}d&g\end{bmatrix}}&E\ar@{-->}[r]^{e^*\delta}&}$ is an $\mathbb{E}$-triangle.

\emph{(2)} Dual of \emph{(1)}.
 \end{lem}
 \begin{proof}
The proof is  model on that of \cite[Proposition 1.20]{LN}. For the convenience of readers, we give the proof here.

 By Lemma \ref{lem1}(1), we get the following commutative diagram made of $\mathbb{E}$-triangles
  $$\xymatrix{\small
    & A\ar[d]_{\tiny\begin{bmatrix}^\exists h\\x\end{bmatrix}} \ar@{=}[r] &A \ar[d]^{x} &\\
  D\ar@{=}[d] \ar[r]^{\tiny\begin{bmatrix}1\\0\end{bmatrix}\ \ \ } &D\oplus B \ar[d]_{\tiny\begin{bmatrix}d&^\exists g'\end{bmatrix}}
   \ar[r]^{\tiny\begin{bmatrix}0&1\end{bmatrix}} &B\ar[d]^{y}\ar@{-->}[r]^{0} &\\
  D\ar[r]^{d}&E\ar[r]^e\ar@{-->}[d]^{e^*\delta}&C\ar@{-->}[r]^{\rho}\ar@{-->}[d]^{\delta}&\\
  &&&   }
  $$
satisfying ${\tiny\begin{bmatrix} h\\x\end{bmatrix}}_*\delta+{\tiny\begin{bmatrix}1\\0\end{bmatrix}}_*\rho=0$, in which we may assume that the middle row is of the form of $\xymatrix@C=1.2cm{ D\ar[r]^{\tiny\begin{bmatrix}1\\0\end{bmatrix}\ \ \ } &D\oplus B \ar[r]^{\quad\tiny\begin{bmatrix}0&1\end{bmatrix}} &B\ar@{-->}[r]^{0} &}$ since  { $y^*\rho=(eg)^*\rho=g^*(e^*\rho)=0$ follows from the fact that $e^*\rho=0$ by \cite[Lemma 3.2]{NP}}. In particular,  $\xymatrix@C=1.2cm{A\ar[r]^{\tiny\begin{bmatrix}h\\x\end{bmatrix}\ \ \ \ \ } &D\oplus B
   \ar[r]^{\ \ \tiny\begin{bmatrix}d&g'\end{bmatrix}} &E\ar@{-->}[r]^{e^*\delta} &}$ is an $\mathbb{E}$-triangle. Note that $eg'=y=eg$,  thus by the exactness of $\xymatrix{\mathcal{C}(B, D)\ar[r]^{\mathcal{C}(B, d)}&\mathcal{C}(B, E)\ar[r]^{\mathcal{C}(B, e)}&\mathcal{C}(B, C),}$ there is $b\in\mathcal{C}(B, D)$ satisfying $db=g'-g$ as $e(g'-g)=0$. For the isomorphism $$\xymatrix{i={\tiny\begin{bmatrix}1&b\\0&1\end{bmatrix}}\colon D\oplus B\xrightarrow{~ \cong~}D\oplus B,}$$ the following diagram is commutative.
   $$\xymatrix{&D\oplus B\ar[dd]^{\cong}_i\ar[dr]^{\tiny\begin{bmatrix}d&g'\end{bmatrix}}&\\
   A\ar[ur]^{\tiny\begin{bmatrix}h\\x\end{bmatrix}}\ar[dr]_{\tiny\begin{bmatrix}h+bx\\x\end{bmatrix}}&&E\\
   &D\oplus B\ar[ur]_{\tiny\begin{bmatrix}d&g\end{bmatrix}}&
   }$$
   Set $f=-(h+bx)$, then $\xymatrix@C=2em{A\ar[r]^{\tiny\begin{bmatrix}-f\\x\end{bmatrix}\ \ \ }&D\oplus B\ar[r]^{\tiny\ \ \begin{bmatrix}d&g\end{bmatrix}}&E\ar@{-->}[r]^{e^*\delta}&}$ is an $\mathbb{E}$-triangle,  $df=-d(h+bx)=-dh-g'x+gx=gx$ and $f_*\delta=-(h+bx)_*\delta=-h_*\delta=-{\tiny\begin{bmatrix}1&0\end{bmatrix}}_*{\tiny\begin{bmatrix}1\\0\end{bmatrix}}_*\rho=\rho$. Hence $(f, g, 1)$ is a morphism of $\mathbb{E}$-triangles.
 \end{proof}

\begin{lem}\label{lema}
\emph{(1)} If $\delta\in\mathbb{E}(C,A)$ and $\xymatrix@C=1.2cm{D\oplus A\ar[r]^{\tiny\begin{bmatrix}1&0\\0&f\end{bmatrix}}&D\oplus B\ar[r]^{\quad\tiny\begin{bmatrix}0&g\end{bmatrix}}&C\ar@{-->}[r]^{{\tiny\begin{bmatrix}0\\1\end{bmatrix}}_*\delta}&}$ is an $\mathbb{E}$-triangle for some $D\in\mathcal{C}$, then $\xymatrix{A\ar[r]^{f}&B\ar[r]^{g}&C\ar@{-->}[r]^{\delta}&}$ is an $\mathbb{E}$-triangle.  Moreover, if

$\xymatrix@C=1.2cm{D\oplus A\ar[r]^{\tiny\begin{bmatrix}1&0\\m&f\end{bmatrix}}&D\oplus B\ar[r]^{\quad\tiny\begin{bmatrix}n&g\end{bmatrix}}&C\ar@{-->}[r]^{{\tiny\begin{bmatrix}0\\1\end{bmatrix}}_*\delta}&}$ is an $\mathbb{E}$-triangle, then $\xymatrix@C=0.8cm{A\ar[r]^{f}&B\ar[r]^{g}&C\ar@{-->}[r]^{\delta}&}$ is an $\mathbb{E}$-triangle.

\emph{(2)} If $\xymatrix{A\ar[r]^{x}&B\ar[r]^{y}&C\ar@{-->}[r]^{\delta}&}$ is an $\mathbb{E}$-triangle, then $\xymatrix{A\ar[r]^{\tiny\begin{bmatrix}0\\x\end{bmatrix}\qquad }&D\oplus B\ar[r]^{\tiny\begin{bmatrix}1&0\\0&y\end{bmatrix}}&
D\oplus C\ar@{-->}[r]^{\ \ \ \ {\tiny\begin{bmatrix}0&1\end{bmatrix}}^*\delta}&}$ is an $\mathbb{E}$-triangle.
\end{lem}
\begin{proof} (1) Assume that the $\mathbb{E}$-triangle $\delta\in\mathbb{E}(C,A)$ is realized as  $\xymatrix{A\ar[r]^{x}&B'\ar[r]^{y}&C\ar@{-->}[r]^{\delta}&}$. { By Definition \ref{def1},} there exists a morphism $ \xymatrix{{{\tiny\begin{bmatrix}a&b\end{bmatrix}}}\colon D\oplus B\ar[r]&B'}$ which gives the morphism of $\mathbb{E}$-triangles
$$\xymatrix@C=1.2cm{D\oplus A\ar[d]^{\tiny\begin{bmatrix}0&1\end{bmatrix}}\ar[r]^{\tiny\begin{bmatrix}1&0\\0&f\end{bmatrix}}&D\oplus B\ar[r]^{\quad\tiny\begin{bmatrix}0&g\end{bmatrix}}\ar@{-->}[d]^{{\tiny\begin{bmatrix}a&b\end{bmatrix}}}&C\ar@{=}[d]
\ar@{-->}[r]^{{\tiny\begin{bmatrix}0\\1\end{bmatrix}}_*\delta}&\\
A\ar[r]^x&B'\ar[r]^y&C\ar@{-->}[r]^\delta&.}$$
Hence {$a=0$}, $x=bf$ and $g=yb$. Meanwhile, { By Definition \ref{def1},} there exists a morphism
${ \tiny\begin{bmatrix}d\\b'\end{bmatrix}}\colon B'\to D\oplus B$ which gives the morphism of $\mathbb{E}$-triangles
$$\xymatrix@C=1.2cm{ A\ar[d]_{\tiny\begin{bmatrix}0\\1\end{bmatrix}}\ar[r]^x& B'\ar[r]^y\ar@{-->}[d]_{\tiny\begin{bmatrix}d\\b'\end{bmatrix}}&C\ar@{=}[d]
\ar@{-->}[r]^{\delta}&\\
D\oplus A\ar[r]^{\tiny\begin{bmatrix}1&0\\0&f\end{bmatrix}}&D\oplus B\ar[r]^{\quad\tiny\begin{bmatrix}0&g\end{bmatrix}}&C\ar@{-->}[r]^{\tiny\begin{bmatrix}0\\1\end{bmatrix}_*\delta}&.}$$
So we have $f=b'x$ and $y=gb'$.
It is easy to check that there exist  morphisms of $\mathbb{E}$-triangles
$$\xymatrix{ A\ar@{=}[d]\ar[r]^x& B'\ar[r]^y\ar[d]_{bb'}&C\ar@{=}[d]
\ar@{-->}[r]^{\delta}&\\
 A\ar[r]^{x}& B'\ar[r]^{y}&C\ar@{-->}[r]^{\delta}&}$$
and
$$\xymatrix@C=1.2cm{D\oplus A\ar@{=}[d]\ar[r]^{\tiny\begin{bmatrix}1&0\\0&f\end{bmatrix}}&D\oplus B\ar[r]^{\quad\tiny\begin{bmatrix}0&g\end{bmatrix}}\ar[d]^{\tiny\begin{bmatrix}1&0\\0&b'b\end{bmatrix}}&C\ar@{=}[d]
\ar@{-->}[r]^{{\tiny\begin{bmatrix}0\\1\end{bmatrix}}_*\delta}&\\
D\oplus A\ar[r]_{\tiny\begin{bmatrix}1&0\\0&f\end{bmatrix}}&D\oplus B\ar[r]_{\quad\tiny\begin{bmatrix}0&g\end{bmatrix}}&C\ar@{-->}[r]^{\tiny\begin{bmatrix}0\\1\end{bmatrix}_*\delta}&.}$$
It follows from \cite[Corollary 3.6]{NP} that $bb'$ and ${\tiny\begin{bmatrix}1&0\\0&b'b\end{bmatrix}}$ are isomorphisms. Hence $b$ is an isomorphism.
 So we have the following commutative diagram
$$\xymatrix@C=1.5cm@R=0.5cm{&B\ar[dr]^g\ar[dd]^b_\cong&\\
A\ar[ur]^f\ar[dr]_x&&C\\
&B'\ar[ur]_y&
}$$
which implies that $\xymatrix{A\ar[r]^{f}&B\ar[r]^{g}&C\ar@{-->}[r]^{\delta}&}$ is an $\mathbb{E}$-triangle.

Moreover, if $\xymatrix@C=1.2cm{D\oplus A\ar[r]^{\tiny\begin{bmatrix}1&0\\m&f\end{bmatrix}}&D\oplus B\ar[r]^{\quad\tiny\begin{bmatrix}n&g\end{bmatrix}}&C\ar@{-->}[r]^{{\tiny\begin{bmatrix}0\\1\end{bmatrix}}_*\delta}&}$ is an $\mathbb{E}$-triangle, then
$$0={\tiny\begin{bmatrix}n&g\end{bmatrix}}{\tiny\begin{bmatrix}1&0\\m&f\end{bmatrix}}={\tiny\begin{bmatrix}n+gm&gf\end{bmatrix}},$$ which implies $n=-gm$. It is easy to check that the following { is a} commutative diagram.

$$\xymatrix{&D\oplus B\ar[dr]^{\tiny\begin{bmatrix}n&g\end{bmatrix}}\ar[dd]_{\tiny\begin{bmatrix}1&0\\-m&1\end{bmatrix}}^\cong&\\
D\oplus A\ar[dr]_{\tiny\begin{bmatrix}1&0\\0&f\end{bmatrix}}\ar[ur]^{\tiny\begin{bmatrix}1&0\\m&f\end{bmatrix}}&&C\\
&D\oplus B\ar[ur]_{\tiny\begin{bmatrix}0&g\end{bmatrix}}&
}$$
Hence $\xymatrix@C=1.2cm{D\oplus A\ar[r]^{\tiny\begin{bmatrix}1&0\\0&f\end{bmatrix}}&D\oplus B\ar[r]^{\quad\tiny\begin{bmatrix}0&g\end{bmatrix}}&C\ar@{-->}[r]^{{\tiny\begin{bmatrix}0\\1\end{bmatrix}}_*\delta}&}$ is an $\mathbb{E}$-triangle, and $\xymatrix{A\ar[r]^{f}&B\ar[r]^{g}&C\ar@{-->}[r]^{\delta}&}$ is an $\mathbb{E}$-triangle by above argument.

(2) Let $\xymatrix{A\ar[r]^{f}&E\ar[r]^{g}&
D\oplus C\ar@{-->}[r]^{\ \ \ \ {\tiny\begin{bmatrix}0&1\end{bmatrix}}_*(0\oplus\delta)}&}$ be  any $\mathbb{E}$-triangle realizing ${\tiny\begin{bmatrix}0&1\end{bmatrix}_*(0\oplus\delta)}$. Then we have the following commutative diagram
$$\xymatrix{0\oplus A\ar[r]^{\tiny\begin{bmatrix}0&0\\0&x\end{bmatrix}}\ar[d]^{\tiny\begin{bmatrix}0&1\end{bmatrix}}&D\oplus B
\ar[r]^{\tiny\begin{bmatrix}1&0\\0&y\end{bmatrix}}\ar[d]^{\tiny\begin{bmatrix}d&b\end{bmatrix}}&D\oplus C\ar@{-->}[r]^{0\oplus \delta}\ar@{=}[d]&\\
A\ar[r]^f&E\ar[r]^g&D\oplus C\ar@{-->}[r]^{\ \ \ \ \tiny\begin{bmatrix}0&1\end{bmatrix}_*(0\oplus\delta)}&.
}$$
It is easy to see that $\xymatrix{{\tiny\begin{bmatrix}0&1\end{bmatrix}}: 0\oplus A\ar[r]&A}$ is an isomorphism, hence
$\xymatrix{{\tiny\begin{bmatrix}d&b\end{bmatrix}}: D\oplus B\ar[r]&E}$ is an isomorphism by \cite[Corollary 3.6]{NP}. Note that $f=bx$, so we have the following commutative diagram
$$\xymatrix{&D\oplus B\ar[dd]^{\cong}_{\tiny\begin{bmatrix}d&b\end{bmatrix}}\ar[dr]^{\tiny\begin{bmatrix}1&0\\0&y\end{bmatrix}}&\\
   A\ar[ur]^{\tiny\begin{bmatrix}0\\x\end{bmatrix}}\ar[dr]_{f}&&D\oplus C\\
   &E\ar[ur]_{g}&
   }$$
   which implies that $\xymatrix{A\ar[r]^{\tiny\begin{bmatrix}0\\x\end{bmatrix}\ \ \ }&D\oplus B\ar[r]^{\tiny\begin{bmatrix}1&0\\0&y\end{bmatrix}}&
D\oplus C\ar@{-->}[r]^{\ \ \ \ {\tiny\begin{bmatrix}0&1\end{bmatrix}}_*(0\oplus\delta)}&}$ is  an $\mathbb{E}$-triangle. Since $${\tiny\begin{bmatrix}1\\0\end{bmatrix}}^*{\tiny\begin{bmatrix}0&1\end{bmatrix}}_*(0\oplus\delta)=0=
{\tiny\begin{bmatrix}1\\0\end{bmatrix}}^*{\tiny\begin{bmatrix}0&1\end{bmatrix}}^*\delta~\textrm{and}~ {\tiny\begin{bmatrix}0\\1\end{bmatrix}}^*{\tiny\begin{bmatrix}0&1\end{bmatrix}}_*(0\oplus\delta)=\delta=
{\tiny\begin{bmatrix}0\\1\end{bmatrix}}^*{\tiny\begin{bmatrix}0&1\end{bmatrix}}^*\delta,$$ we have ${\tiny\begin{bmatrix}0&1\end{bmatrix}}_*(0\oplus\delta)={\tiny\begin{bmatrix}0&1\end{bmatrix}}^*\delta$, as desired.
   \end{proof}
The next lemma parallels \cite[Lemma 3.14]{NP} and its dual.

 \begin{lem}\label{lemc}
\emph{(1)} Let $\xymatrix@C=2em{A\ar[r]^f&B\ar[r]^{f'}&C\ar@{-->}[r]^{\delta_f}&,}$ $\xymatrix@C=2em{B\ar[r]^g&D\ar[r]^{g'}&F\ar@{-->}[r]^{\delta_g}&}$ and
 $$\xymatrix@C=2em{A\ar[r]^h&D\ar[r]^{h'}&E\ar@{-->}[r]^{\delta_h}&}$$ be any triplet of $\mathbb{E}$-triangles satisfying $h=gf$. Then there are morphisms $d$ and $e$ in $\mathcal{C}$ which make the diagram
 $$\xymatrix{A\ar[r]^{f}\ar@{=}[d]&B\ar[r]^{f'}\ar[d]_{g}&C\ar[d]^d\ar@{-->}[r]^{\delta_f}&\\
A\ar[r]^h&D\ar[r]^{h}\ar[d]_{g'}&E\ar[d]^e\ar@{-->}[r]^{\delta_h}&\\
&F\ar@{-->}[d]^{\delta_g}\ar@{=}[r]&F\ar@{-->}[d]^{f'_*(\delta_g)}&\\
&&&}$$
commutative, and satisfying the following compatibilities.

\emph{(i)} $\xymatrix@C=2em{C\ar[r]^d&E\ar[r]^{e}&F\ar@{-->}[r]^{f'_*(\delta_g)}&}$ is an $\mathbb{E}$-triangle,

\emph{(ii)} $d^*(\delta_h)=\delta_f$,

\emph{(iii)} $e^*(\delta_g)=f_*(\delta_h)$,

\emph{(iv)} $\xymatrix@C=2em{B\ar[r]^{\tiny\begin{bmatrix}g\\f'\end{bmatrix}\ \ \ }&D\oplus C\ar[r]^{\tiny\ \ \begin{bmatrix}h&-d\end{bmatrix}}&E\ar@{-->}[r]^{f_*(\delta_h)}&}$ is an $\mathbb{E}$-triangle.

\emph{(2)} Dual of \emph{(1)}.\end{lem}

\begin{proof} (1) There is a morphism $d: C\rightarrow E$ such that $d^*(\delta_h)=\delta_f$ and $\xymatrix@C=2em{B\ar[r]^{\tiny\begin{bmatrix}g\\f'\end{bmatrix}\ \ \ }&D\oplus C\ar[r]^{\tiny\ \ \begin{bmatrix}h&-d\end{bmatrix}}&E\ar@{-->}[r]^{f_*(\delta_h)}&}$ is an $\mathbb{E}$-triangle by Lemma \ref{lemb}(2). Now we apply Lemma \ref{lemb}(2) to the diagram
$$\xymatrix{B\ar[r]^{\tiny\begin{bmatrix}g\\f'\end{bmatrix}}\ar@{=}[d]&D\oplus C\ar[r]^{\tiny\begin{bmatrix}h&-d\end{bmatrix}}\ar[d]_{\tiny\begin{bmatrix}1&0\end{bmatrix}}&E\ar@{-->}[d]^{e}\ar@{-->}[r]^{f_*(\delta_h)}&\\
B\ar[r]^g&D\ar[r]^{g'}&F\ar@{-->}[r]^{\delta_g}&
}$$
there is a morphism $e: E\rightarrow F$ such that $e^*(\delta_g)=f_*(\delta_h)$ and $\xymatrix@C=2em{D\oplus C\ar[r]^{\tiny\begin{bmatrix}1&0\\h&-d\end{bmatrix}\ \ \ }&D\oplus E\ar[r]^{\tiny\ \ \begin{bmatrix}g'&-e\end{bmatrix}}&F\ar@{-->}[r]^{\tiny\begin{bmatrix}g\\f'\end{bmatrix}_*(\delta_g)}&}$ is an $\mathbb{E}$-triangle. Since ${\tiny\begin{bmatrix}1&0\end{bmatrix}}_*{\tiny\begin{bmatrix}g\\f'\end{bmatrix}}_*(\delta_g)=g_*\delta_g=0$ and ${\tiny\begin{bmatrix}0&1\end{bmatrix}}_*{\tiny\begin{bmatrix}g\\f'\end{bmatrix}}_*(\delta_g)=f'_*(\delta_g)$, we have ${\tiny\begin{bmatrix}g\\f'\end{bmatrix}}_*(\delta_g)= {\tiny\begin{bmatrix}0\\1\end{bmatrix}}_*f'_*(\delta_g)$. It follows from Lemma \ref{lema}(1) that $\xymatrix@C=2em{C\ar[r]^{-d}&E\ar[r]^{-e}&F\ar@{-->}[r]^{f'_*(\delta_g)}&}$ is an $\mathbb{E}$-triangle. It is easy to check that $\xymatrix@C=2em{C\ar[r]^{d}&E\ar[r]^{e}&F\ar@{-->}[r]^{f'_*(\delta_g)}&}$ is an $\mathbb{E}$-triangle, as desired.
\end{proof}

\begin{prop}\label{thm} Let $\xi$ be a class of $\mathbb{E}$-triangles satisfying the conditions (1) and (2) in Definition \ref{def:proper class}.


Then $\xi$ is saturated if and only if for the diagram in Lemma \ref{lem1}(2)
$$\xymatrix{
     A\ar[d]_{x_2} \ar[r]^{x_1} & B_1 \ar[d]^{m_2}\ar[r]^{y_1}&C_1\ar@{=}[d] \\
  B_2 \ar[d]_{y_2} \ar[r]^{m_1} & M \ar[d]^{e_2} \ar[r]^{e_1} & C_1 \\
  C_2 \ar@{=}[r] & C_2 &   }
  $$
if the $\mathbb{E}$-triangles $\xymatrix@C=2em{A\ar[r]^{x_2}&B_2\ar[r]^{y_2}&C_2\ar@{-->}[r]^{\delta_2}&}$ and $\xymatrix@C=2em{B_2\ar[r]^{m_1}&M\ar[r]^{e_1}&C_1\ar@{-->}[r]^{x_{2*}\delta_1}&}$ are in $\xi$, then the $\mathbb{E}$-triangle $\xymatrix@C=2em{A\ar[r]^{x_1}&B_1\ar[r]^{y_1}&C_1\ar@{-->}[r]^{\delta_1}&}$ is also in $\xi$.

 \end{prop}
  \begin{proof}
 We only prove the `only if' part, the proof of `if' part is similar.
 { Now we assume that $\xi$ is saturated.}
By (ET4)  and Lemma \ref{lemc}(1), there exists the following commutative diagram
 $$\xymatrix{A\ar[r]^{x_2}\ar@{=}[d]&B_2\ar[r]^{y_2}\ar[d]_{m_1}&C_2\ar[d]^d\ar@{-->}[r]^{\delta_2}&\\
A\ar[r]^f&M\ar[r]^{g}\ar[d]_{e_1}&W\ar[d]^h\ar@{-->}[r]^{\rho}&\\
&C_1\ar@{-->}[d]^{x_{2*}\delta_1}\ar@{=}[r]&C_1\ar@{-->}[d]^{y_{2*}x_{2*}\delta_1}&\\
&&&
}$$
in $\mathcal{C}$, and an $\mathbb{E}$-extension $\rho\in\mathbb{E}(W, A)$ realized by $\xymatrix{A\ar[r]^f&M\ar[r]^{g}&W},$
which satisfy the following compatibilities.

(i) $\xymatrix{C_2\ar[r]^d&W\ar[r]^{h}&C_1}$ realizes $y_{2*}x_{2*}\delta_1$,

(ii) $d^*\rho=\delta_2$,

(iii) $x_{2*}\rho=h^*x_{2*}\delta_1$,

(iv) $\xymatrix@C=2em{B_2\ar[r]^{\tiny\begin{bmatrix}m_1\\y_2\end{bmatrix}\ \ \ }&M\oplus C_2\ar[r]^{\tiny\ \ \begin{bmatrix}g&-d\end{bmatrix}}&W\ar@{-->}[r]^{x_{2*}\rho}&}$ is an $\mathbb{E}$-triangle.

\noindent Note that $\xymatrix@C=2em{B_2\ar[r]^{m_1}&M\ar[r]^{e_1}&C_1\ar@{-->}[r]^{x_{2*}\delta_1}&}$ is an $\mathbb{E}$-triangle in $\xi$, then so is $$\xymatrix@C=1.2cm{B_2\ar[r]^{\tiny\begin{bmatrix}m_1\\y_2\end{bmatrix}\ \ \ }&M\oplus C_2\ar[r]^{\tiny\ \ \begin{bmatrix}g&-d\end{bmatrix}}&W\ar@{-->}[r]^{x_{2*}\rho}&}$$ by (iii) because $\xi$ is closed under base change.
It is easy to check that $$\xymatrix@C=2em{B_2\ar[r]^{\tiny\begin{bmatrix}m_1\\-y_2\end{bmatrix}\ \ \ }&M\oplus C_2\ar[r]^{\tiny\ \ \begin{bmatrix}g&d\end{bmatrix}}&W\ar@{-->}[r]^{x_{2*}\rho}&}$$ is an $\mathbb{E}$-triangle in $\xi$.
It follows from Lemma \ref{lema}(2) that
$$\xymatrix{ A\ar[r]^{\tiny\begin{bmatrix}0\\x_2\end{bmatrix}\ \ \ } &M\oplus B_2
   \ar[r]^{\tiny\begin{bmatrix}1&0\\0&y_2\end{bmatrix}} &M\oplus C_2\ar@{-->}[r]^{\ \ \ {\tiny\begin{bmatrix}0&1\end{bmatrix}}^*\delta_2} &
   }$$ is an $\mathbb{E}$-triangle, it is an $\mathbb{E}$-triangle in $\xi$ because $\xymatrix@C=2em{A\ar[r]^{x_2}&B_2\ar[r]^{y_2}&C_2\ar@{-->}[r]^{\delta_2}&}$ is an $\mathbb{E}$-triangle in $\xi$. We have $g^*x_{2*}\rho=x_{2*}g^*\rho=0$. One can prove that ${\tiny\begin{bmatrix}1\\0\end{bmatrix}}^*{\tiny\begin{bmatrix}g&d\end{bmatrix}}^*\rho=g^*\rho=0$ and ${\tiny\begin{bmatrix}0\\1\end{bmatrix}}^*{\tiny\begin{bmatrix}g&d\end{bmatrix}}^*\rho=d^*\rho=\delta_2$, which implies { ${\tiny\begin{bmatrix}g&d\end{bmatrix}}^*\rho={\tiny\begin{bmatrix}0&1\end{bmatrix}}^*\delta_2$}.
  It is straightforward to show that the following diagram is commutative.
 $$\xymatrix{\small
    & B_2\ar[d]_{\tiny\begin{bmatrix}m_1\\-1\end{bmatrix}} \ar@{=}[r] & B_2 \ar[d]^{\tiny\begin{bmatrix}m_1\\-y_2\end{bmatrix}} &\\
  A\ar@{=}[d] \ar[r]^{\tiny\begin{bmatrix}0\\x_2\end{bmatrix}\ \ \ \ \ } &M\oplus B_2 \ar[d]_{\tiny\begin{bmatrix}1&m_1\end{bmatrix}}
   \ar[r]^{\tiny\begin{bmatrix}1&0\\0&y_2\end{bmatrix}} & M\oplus C_2\ar[d]_{\tiny\begin{bmatrix}g&d\end{bmatrix}}\ar@{-->}[r]_{\ \ \ \ \ \ {\tiny\begin{bmatrix}0&1\end{bmatrix}}^*\delta_2} &\\
  A\ar[r]^{f}&M\ar[r]^g\ar@{-->}[d]^0&W\ar@{-->}[r]^{\rho}\ar@{-->}[d]^{x_{2*}\rho}&\\
  &&&   }
  $$
Hence
   $\xymatrix{A\ar[r]^f&M\ar[r]^g&W\ar@{-->}[r]^\rho&}$ is an $\mathbb{E}$-triangle in $\xi$ since $\xi$ is closed under saturated.
   Since $f=m_1x_2=m_2x_1$, applying (ET3) to
   $$\xymatrix{A\ar[r]^{x_1}\ar@{=}[d]&B_1\ar[r]^{y_1}\ar[d]^{m_2}&C_1\ar@{-->}[r]^{\delta_1}&\\
   A\ar[r]^f&M\ar[r]^g&W\ar@{-->}[r]^\rho&}$$
   we obtain a morphism $c\in \mathcal{C}(C_1, W)$ satisfying $\delta_1=c^*\rho$. Hence $\xymatrix@C=2em{A\ar[r]^{x_1}&B_1\ar[r]^{y_1}&C_1\ar@{-->}[r]^{\delta_1}&}$ is an $\mathbb{E}$-triangle in $\xi$.
\end{proof}

We are ready to prove Theorem \ref{thma}.

\emph{Proof of Theorem \ref{thma}}. ``$\Rightarrow$'' It is easy to check that $(\mathcal{C}, \mathbb{E}_\xi, \mathfrak{s}_\xi)$ is an extriangulated category by the definition of proper class and the proof of Proposition \ref{thm}.

``$\Leftarrow$'' { Assume that} $(\mathcal{C}, \mathbb{E}_\xi, \mathfrak{s}_\xi)$ is an extriangulated category. { It} is easy to check that $\xi$ satisfies the conditions (1) and (2) in Definition \ref{def:proper class}. Next we claim that $\xi$ is saturated. Consider the diagram in Lemma \ref{lem1}(2)
$$\xymatrix{
     A\ar[d]_{x_2} \ar[r]^{x_1} & B_1 \ar[d]^{m_2}\ar[r]^{y_1}&C_1\ar@{=}[d] \\
  B_2 \ar[d]_{y_2} \ar[r]^{m_1} & M \ar[d]^{e_2} \ar[r]^{e_1} & C_1 \\
  C_2 \ar@{=}[r] & C_2 &   }
  $$
 where the $\mathbb{E}$-triangles $\xymatrix@C=2em{A\ar[r]^{x_2}&B_2\ar[r]^{y_2}&C_2\ar@{-->}[r]^{\delta_2}&}$ and $\xymatrix@C=2em{B_2\ar[r]^{m_1}&M\ar[r]^{e_1}&C_1\ar@{-->}[r]^{x_{2*}\delta_1}&}$ are in $\xi$.

 Since $(\mathcal{C}, \mathbb{E}_\xi, \mathfrak{s}_\xi)$ is an extriangulated category,  there exists the following commutative diagram
 $$\xymatrix{A\ar[r]^{x_2}\ar@{=}[d]&B_2\ar[r]^{y_2}\ar[d]_{m_1}&C_2\ar[d]^d\ar@{-->}[r]^{\delta_2}&\\
A\ar[r]^f&M\ar[r]^{g}\ar[d]_{e_1}&W\ar[d]^h\ar@{-->}[r]^{\rho}&\\
&C_1\ar@{-->}[d]^{x_{2*}\delta_1}\ar@{=}[r]&C_1\ar@{-->}[d]^{y_{2*}x_{2*}\delta_1}&\\
&&&
}$$
 where $\xymatrix{A\ar[r]^f&M\ar[r]^{g}&W\ar@{-->}[r]^\rho&}$ is an $\mathbb{E}$-triangle in $\xi$ { by (ET4) of  the extriangulated category} $(\mathcal{C}, \mathbb{E}_\xi, \mathfrak{s}_\xi)$.  Since $f=m_1x_2=m_2x_1$, applying (ET3) of  extriangulated category $(\mathcal{C}, \mathbb{E}, \mathfrak{s})$ to
   $$\xymatrix{A\ar[r]^{x_1}\ar@{=}[d]&B_1\ar[r]^{y_1}\ar[d]^{m_2}&C_1\ar@{-->}[r]^{\delta_1}&\\
   A\ar[r]^f&M\ar[r]^g&W\ar@{-->}[r]^\rho&}$$
   we obtain a morphism $c\in \mathcal{C}(C_1, W)$ satisfying $\delta_1=c^*\rho$, hence $\xymatrix@C=2em{A\ar[r]^{x_1}&B_1\ar[r]^{y_1}&C_1\ar@{-->}[r]^{\delta_1}&}$ is an $\mathbb{E}$-triangle in $\xi$. So $\xi$ is saturated by Proposition \ref{thm}.\hfill$\Box$

\section{\bf $\xi$-$\mathcal{G}$projective objects and some fundamental properties}
Throughout the section, we assume that $\xi$ is a proper class of $\mathbb{E}$-triangles in an extriangulated category $(\mathcal{C}, \mathbb{E}, \mathfrak{s})$.

\subsection{The definition of $\xi$-$\mathcal{G}$projective objects} The following concept is inspired from $\xi$-projective objects with respect to a proper class of triangles in a triangulated category.

\begin{definition} {\rm An object $P\in\mathcal{C}$  is called {\it $\xi$-projective}  if for any $\mathbb{E}$-triangle $$\xymatrix{A\ar[r]^x& B\ar[r]^y& C \ar@{-->}[r]^{\delta}& }$$ in $\xi$, the induced sequence of abelian groups $\xymatrix@C=0.6cm{0\ar[r]& \mathcal{C}(P,A)\ar[r]& \mathcal{C}(P,B)\ar[r]&\mathcal{C}(P,C)\ar[r]& 0}$ is exact. Dually, we have the definition of {\it $\xi$-injective}.}
\end{definition}

We denote $\mathcal{P(\xi)}$ (respectively $\mathcal{I(\xi)}$) the class of $\xi$-projective (respectively $\xi$-injective) objects of $\mathcal{C}$. It follows from the definition that this subcategory $\mathcal{P}(\xi)$ and $\mathcal{I}(\xi)$ are full, additive, closed under isomorphisms and direct summands.

 An extriangulated  category $(\mathcal{C}, \mathbb{E}, \mathfrak{s})$ is said to  have {\it  enough
$\xi$-projectives} \ (respectively {\it  enough $\xi$-injectives}) provided that for each object $A$ there exists an $\mathbb{E}$-triangle $\xymatrix@C=2.1em{K\ar[r]& P\ar[r]&A\ar@{-->}[r]& }$ (respectively $\xymatrix@C=2em{A\ar[r]& I\ar[r]& K\ar@{-->}[r]&}$) in $\xi$ with $P\in\mathcal{P}(\xi)$ (respectively $I\in\mathcal{I}(\xi)$).

\begin{lem}\label{lem4}  If $\mathcal{C}$ has enough $\xi$-projectives, then an $\mathbb{E}$-triangle $\xymatrix@C=2em{A\ar[r]&B\ar[r]& C\ar@{-->}[r]&}$ is in $\xi$ if and only if  induced sequence of abelian groups $\xymatrix@C=1.7em{0\ar[r]&\mathcal{C}(P,A)\ar[r]&\mathcal{C}(P,B)\ar[r]&\mathcal{C}(P,C)\ar[r]&0}$ is exact for all $P\in\mathcal{P}(\xi)$.
\end{lem}

{ \proof The proof is similar to \cite[Lemma 4.2]{Bel1}.  \qed}

\begin{prop} {\rm (Schanuel's lemma).} If $\xymatrix@C=2em{K_i\ar[r]^{x_i}&P_i\ar[r]^{y_i}& C\ar@{-->}[r]^{\delta_i}&}$ are $\mathbb{E}$-triangles in $\xi$ with $P_i\in\mathcal{P}(\xi), i=1, 2$, then $K_1\oplus P_2\cong K_2\oplus P_1$.
\end{prop}
{ \proof The proof is similar to \cite[Proposition 4.4]{Bel1}.  \qed}

If $\xymatrix@C=2em{K\ar[r]^{x}&P\ar[r]^{y}& C\ar@{-->}[r]^{\delta}&}$ is an $\mathbb{E}$-triangles in $\xi$ with $P\in\mathcal{P}(\xi)$, then we call the object $K$ a  {\it first $\xi$-syzygy} of $C$. An {\it $n${\rm th} $\xi$-syzygy} of $C$ is defined as usual by induction. By Schanuel's lemma any two $\xi$-syzygies of $C$ are isomorphic modulo $\xi$-projectives.

The {\it $\xi$-projective dimension} $\xi$-${\rm pd} A$ of $A\in\mathcal{C}$ is defined inductively.
 If $A\in\mathcal{P}(\xi)$, then define $\xi$-${\rm pd} A=0$.
Next if $\xi$-${\rm pd} A>0$, define $\xi$-${\rm pd} A\leqslant n$ if there exists an $\mathbb{E}$-triangle
 $K\to P\to A\dashrightarrow$  in $\xi$ with $P\in \mathcal{P}(\xi)$ and $\xi$-${\rm pd} K\leqslant n-1$.
Finally we define $\xi$-${\rm pd} A=n$ if $\xi$-${\rm pd} A\leqslant n$ and $\xi$-${\rm pd} A\nleq n-1$. Of course we set $\xi$-${\rm pd} A=\infty$, if $\xi$-${\rm pd} A\neq n$ for all $n\geqslant 0$.

Dually we can define the {\it $\xi$-injective dimension}  $\xi$-${\rm id} A$ of an object $A\in\mathcal{C}$.


\begin{definition} {\rm An unbounded complex $\mathbf{X}$ is called {\it $\xi$-exact} complex  if $\mathbf{X}$ is a diagram $$\xymatrix@C=2em{\cdots\ar[r]&X_1\ar[r]^{d_1}&X_0\ar[r]^{d_0}&X_{-1}\ar[r]&\cdots}$$ in $\mathcal{C}$ such that for each integer $n$, there exists an $\mathbb{E}$-triangle $\xymatrix@C=2em{K_{n+1}\ar[r]^{g_n}&X_n\ar[r]^{f_n}&K_n\ar@{-->}[r]^{\delta_n}&}$ in $\xi$ and $d_n=g_{n-1}f_n$.}
\end{definition}

\begin{definition} {\rm Let $\mathcal{W}$ be a class of objects in $\mathcal{C}$. An $\mathbb{E}$-triangle $\xymatrix@C=2em{A\ar[r]& B\ar[r]& C\ar@{-->}[r]& }$ in $\xi$ is called to be
{\it $\mathcal{C}(-,\mathcal{W})$-exact} (respectively
{\it $\mathcal{C}(\mathcal{W},-)$-exact}) if for any $W\in\mathcal{W}$, the induced sequence of abelian group $\xymatrix@C=2em{0\ar[r]&\mathcal{C}(C,W)\ar[r]&\mathcal{C}(B,W)\ar[r]&\mathcal{C}(A,W)\ar[r]& 0}$ (respectively \\ $\xymatrix@C=2em{0\ar[r]&\mathcal{C}(W,A)\ar[r]&\mathcal{C}(W,B)\ar[r]&\mathcal{C}(W,C)\ar[r]&0}$) is exact in ${\rm Ab}$}.
\end{definition}

\begin{definition} {\rm Let $\mathcal{W}$ be a class of objects in $\mathcal{C}$. A complex $\mathbf{X}$ is called {\it $\mathcal{C}(-,\mathcal{W})$-exact} (respectively
{\it $\mathcal{C}(\mathcal{W},-)$-exact}) if it is a $\xi$-exact complex
$$\xymatrix@C=2em{\cdots\ar[r]&X_1\ar[r]^{d_1}&X_0\ar[r]^{d_0}&X_{-1}\ar[r]&\cdots}$$ in $\mathcal{C}$ such that  { there exists a  $\mathcal{C}(-,\mathcal{W})$-exact (respectively
 $\mathcal{C}(\mathcal{W},-)$-exact) $\mathbb{E}$-triangle $$\xymatrix@C=2em{K_{n+1}\ar[r]^{g_n}&X_n\ar[r]^{f_n}&K_n\ar@{-->}[r]^{\delta_n}&}$$ in $\xi$  and $d_n=g_{n-1}f_n$ for each integer $n$}.

 A $\xi$-exact complex $\mathbf{X}$ is called {\it complete $\mathcal{P}(\xi)$-exact} (respectively {\it complete $\mathcal{I}(\xi)$-exact}) if it is $\mathcal{C}(-,\mathcal{P}(\xi))$-exact (respectively
 $\mathcal{C}(\mathcal{I}(\xi),-)$-exact).}
\end{definition}

\begin{definition}  {\rm A  {\it complete $\xi$-projective resolution}  is a complete $\mathcal{P}(\xi)$-exact complex\\ $$\xymatrix@C=2em{\mathbf{P}:\cdots\ar[r]&P_1\ar[r]^{d_1}&P_0\ar[r]^{d_0}&P_{-1}\ar[r]&\cdots}$$ in $\mathcal{C}$ such that $P_n$ is $\xi$-projective for each integer $n$. Dually,   a  {\it complete $\xi$-injective coresolution}  is a complete $\mathcal{I}(\xi)$-exact complex $$\xymatrix@C=2em{\mathbf{I}:\cdots\ar[r]&I_1\ar[r]^{d_1}&I_0\ar[r]^{d_0}&I_{-1}\ar[r]&\cdots}$$ in $\mathcal{C}$ such that $I_n$ is $\xi$-injective for each integer $n$.}
\end{definition}

\begin{definition} {\rm  Let $\mathbf{P}$ be a complete $\xi$-projective resolution in $\mathcal{C}$. So for each integer $n$, there exists a $\mathcal{C}(-, \mathcal{P}(\xi))$-exact $\mathbb{E}$-triangle $\xymatrix@C=2em{K_{n+1}\ar[r]^{g_n}&P_n\ar[r]^{f_n}&K_n\ar@{-->}[r]^{\delta_n}&}$ in $\xi$. The objects $K_n$ are called {\it $\xi$-$\mathcal{G}$projective} for each integer $n$. Dually if  $\mathbf{I}$ is a complete $\xi$-injective  coresolution in $\mathcal{C}$, there exists a  $\mathcal{C}(\mathcal{I}(\xi), -)$-exact $\mathbb{E}$-triangle $\xymatrix@C=2em{K_{n+1}\ar[r]^{g_n}&I_n\ar[r]^{f_n}&K_n\ar@{-->}[r]^{\delta_n}&}$ in $\xi$ for each integer $n$. The objects $K_n$ are called {\it $\xi$-$\mathcal{G}$injective} for each integer $n$.}
{ We denote  by $\mathcal{GP}(\xi)$ the class of $\xi$-$\mathcal{G}$projective objects in $\mathcal{C}$.}
\end{definition}

\begin{Ex} \emph{(1)} Assume that $(\mathcal{C}, \mathbb{E}, \mathfrak{s})$ is the category of left $R$-modules. If $\xi$ is the class of all $\mathbb{E}$-triangles, then $\xi$-$\mathcal{G}$projective objects defined
in here coincides with the earlier one given by Enochs and Jenda in \cite{EJ1}.

\emph{(2)} If $\mathcal{C}$ is a triangulated category and $\xi$ is { a proper class  of triangles} which is closed under suspension \emph{(see \cite[Section 2.2]{Bel1})}, then $\xi$-$\mathcal{G}$projective objects defined
in here coincides with the earlier one given by Asadollahi and Salarian in \cite{AS1}.
\end{Ex}

\subsection{Some fundamental properties of $\xi$-$\mathcal{G}$projective objects} Throughout this section, we always assume that the extriangulated category $(\mathcal{C}, \mathbb{E}, \mathfrak{s})$ has enough $\xi$-projectives and enough $\xi$-injectives. The following observation is useful in this section.



\begin{lem}\label{lem5}
\emph{(1)} Let $\mathcal{W}$ be a subcategory of $\mathcal{C}$,  $$\xymatrix@C=2em{A_1\ar[r]^{x_1}&B_1\ar[r]^{y_1}&C\ar@{-->}[r]^{\delta_1}&}~and~ \xymatrix@C=2em{A_2\ar[r]^{x_2}&B_2\ar[r]^{y_2}&C\ar@{-->}[r]^{\delta_2}&}$$ be $\mathbb{E}$-triangles in $\xi$. Consider the following commutative diagram
in $\mathcal{C}$
$$\xymatrix{
    & A_2\ar[d]_{m_2} \ar@{=}[r] & A_2 \ar[d]^{x_2} \\
  A_1 \ar@{=}[d] \ar[r]^{m_1} & M \ar[d]_{e_2} \ar[r]^{e_1} & B_2\ar[d]^{y_2} \\
  A_1 \ar[r]^{x_1} & B_1\ar[r]^{y_1} & C   }
  $$
  which satisfies $\mathfrak{s}(y^*_2\delta_1)=\xymatrix@C=2em{[A_1\ar[r]^{m_1}&M\ar[r]^{e_1}&B_2]}$ and $\mathfrak{s}(y^*_1\delta_2)=\xymatrix@C=2em{[A_2\ar[r]^{m_2}&M\ar[r]^{e_2}&B_1].}$

  If  $\xymatrix@C=2em{A_2\ar[r]^{x_2}&B_2\ar[r]^{y_2}&C\ar@{-->}[r]^{\delta_2}&}$ and
  $\xymatrix@C=2em{A_1\ar[r]^{m_1}&M\ar[r]^{e_1}&B_2\ar@{-->}[r]^{y_2^*\delta_1}&}$ are $\mathcal{C}(-, \mathcal{W})$-exact, then

  $\xymatrix@C=2em{A_1\ar[r]^{x_1}&B_1\ar[r]^{y_1}&C\ar@{-->}[r]^{\delta_1}&}$ is  $\mathcal{C}(-, \mathcal{W})$-exact.

 \emph{(2)} If $\xymatrix@C=2em{A\ar[r]^{x}&B\ar[r]^{y}&C\ar@{-->}[r]^{\delta}&}$  is an $\mathbb{E}$-triangle in $\xi$ with $C\in\mathcal{GP}(\xi)$, then it is  $\mathcal{C}(-,\mathcal{P}(\xi))$-exact.

 \emph{(3)} If $\xymatrix@C=2em{A\ar[r]^{x}&B\ar[r]^{y}&C\ar@{-->}[r]^{\delta}&}$  is an $\mathbb{E}$-triangle in $\xi$ where $C$ is a direct summand of $\xi$-$\mathcal{G}$projective, then it is $\mathcal{C}(-,\mathcal{P}(\xi))$-exact.
  \end{lem}
\begin{proof} (1) Since $\xymatrix@C=2em{A_2\ar[r]^{x_2}&B_2\ar[r]^{y_2}&C\ar@{-->}[r]^{\delta_2}&}$ is an $\mathbb{E}$-triangle, there exists an exact sequence of abelian groups $\xymatrix{\mathcal{C}(C, W)\ar[r]^{\mathcal{C}(y_2, W)}&\mathcal{C}(B_2, W)\ar[r]^{\mathcal{C}(x_2, W)}&\mathcal{C}(A_2, W)\ar[r]^{(\delta_{2})_W^\sharp}&\mathbb{E}(C, W)\ar[r]^{\mathbb{E}(y_2, W)}&\mathbb{E}(B_2, W)}$ for any $W\in\mathcal{W}$ by \cite[Proposition 3.3]{NP}, which implies that the morphisms $\mathcal{C}(y_2, W)$
and $\mathbb{E}(y_2, W)$ are monic as $\xymatrix@C=2em{A_2\ar[r]^{x_2}&B_2\ar[r]^{y_2}&C\ar@{-->}[r]^{\delta_2}&}$ is $\mathcal{C}(-,\mathcal{W})$-exact.
Similarly, we can show that $\mathcal{C}(e_1, W)$
and $\mathbb{E}(e_1, W)$ are monic for any $W\in\mathcal{W}$. It follows from \cite[Proposition 3.3]{NP}, there is a commutative diagram of abelian groups for any $W\in\mathcal{W}$.
$$\xymatrix{\mathcal{C}(C, W)\ar[r]^{\mathcal{C}(y_1, W)}\ar[d]^{\mathcal{C}(y_2, W)}&\mathcal{C}(B_1, W)\ar[r]^{\mathcal{C}(x_1, W)}\ar[d]^{\mathcal{C}(e_2, W)}&\mathcal{C}(A_1, W)\ar[r]^{(\delta_1)^\sharp_W}\ar@{=}[d]&\mathbb{E}(C, W)\ar[r]^{\mathbb{E}(y_1, W)}\ar[d]^{\mathbb{E}(y_2, W)}&\mathbb{E}(B_1, W)\ar[d]^{\mathbb{E}(e_2, W)}\\
\mathcal{C}(B_2, W)\ar[r]^{\mathcal{C}(e_1, W)}&\mathcal{C}(M, W)\ar[r]^{\mathcal{C}(m_1, W)}&\mathcal{C}(A_1, W)\ar[r]^{(y_2^*\delta_1)^\sharp_W}&\mathbb{E}(B_2, W)\ar[r]^{\mathbb{E}(e_1, W)}&\mathbb{E}(M, W)}$$
It is easy to see that $\mathcal{C}(y_1, W)$
and $\mathbb{E}(y_1, W)$ are monic, which implies   that $$\xymatrix@C=2em{A_1\ar[r]^{x_1}&B_1\ar[r]^{y_1}&C\ar@{-->}[r]^{\delta_1}&}$$ is $\mathcal{C}(-,\mathcal{W})$-exact.

(2) If $\xymatrix@C=2em{A\ar[r]^{x}&B\ar[r]^{y}&C\ar@{-->}[r]^{\delta}&}$  is an $\mathbb{E}$-triangle in $\xi$ with $C\in\mathcal{GP}(\xi)$, then there exists a $\mathcal{C}(-,\mathcal{P}(\xi))$-exact $\mathbb{E}$-triangle $\xymatrix@C=2em{G\ar[r]^{f}&P\ar[r]^{g}&C\ar@{-->}[r]^{\rho}&}$ in $\xi$ with $P\in\mathcal{P}(\xi)$ and $G\in\mathcal{GP}(\xi)$. Hence there exists a commutative diagram
$$\xymatrix{
    & G\ar[d]_{f_1} \ar@{=}[r] & G\ar[d]^{f} \\
  A \ar@{=}[d] \ar[r]^{x_1} & M \ar[d]_{g_1} \ar[r]^{y_1} & P\ar[d]^{g} \\
  A \ar[r]^{x} & B\ar[r]^{y} & C   }
  $$
  which satisfies $\mathfrak{s}(g^*\delta)=\xymatrix@C=2em{[A\ar[r]^{x_1}&M\ar[r]^{y_1}&P]}$ and $\mathfrak{s}(y^*\rho)=\xymatrix@C=2em{[G\ar[r]^{f_1}&M\ar[r]^{g_1}&B].}$
{ Since $P\in\mathcal{P}(\xi)$, $g$ factors through $y$.
We know that the $\mathbb{E}$-triangle $\xymatrix{A \ar[r]^{x_1} & M  \ar[r]^{y_1} & P\ar@{-->}[r]^{g^*\delta}&}$ is split by  \cite[Corollary 3.5]{NP},} hence it is a
$\mathcal{C}(-,\mathcal{P}(\xi))$-exact $\mathbb{E}$-triangle in $\xi$. So $$\xymatrix@C=2em{A\ar[r]^{x}&B\ar[r]^{y}&C\ar@{-->}[r]^{\delta}&}$$  is
  $\mathcal{C}(-,\mathcal{P}(\xi))$-exact by (1).

   (3) Suppose that $G=C\oplus C'$ with  $G\in\mathcal{GP}(\xi)$, then we have the following commutative diagram
  $$\xymatrix{
    & C'\ar[d]_{f_1} \ar@{=}[r] & C'\ar[d]^{{\tiny\begin{bmatrix}0\\1\end{bmatrix}}} \\
  A \ar@{=}[d] \ar[r]^{x_1} & M \ar[d]_{g_1} \ar[r]^{y_1} & G\ar[d]^{\tiny\begin{bmatrix}1&0\end{bmatrix}} \\
  A \ar[r]^{x} & B\ar[r]^{y} & C   }
  $$
made of $\mathbb{E}$-triangles in $\xi$, which satisfies  $\mathfrak{s}({\tiny\begin{bmatrix}1&0\end{bmatrix}}^*\delta)=\xymatrix@C=2em{[A\ar[r]^{x_1}&M\ar[r]^{y_1}&G].}$ Note that the second horizontal  is $\mathcal{C}(-,\mathcal{P}(\xi))$-exact by (2) and the third vertical is $\mathcal{C}(-,\mathcal{P}(\xi))$-exact, then so is the third horizontal by (1).
\end{proof}

In addition, we assume the following condition for { the rest of the section} (see \cite[Condition 5.8]{NP}).

\begin{cond} \label{cond:4.11} \emph{({\rm Condition (WIC)})}  Consider the following conditions.

\emph{(1)} Let $f\in\mathcal{C}(A, B), g\in\mathcal{C}(B, C)$ be any composable pair of morphisms. If $gf$ is an inflation, then so is $f$.

\emph{(2)} Let $f\in\mathcal{C}(A, B), g\in\mathcal{C}(B, C)$ be any composable pair of morphisms. If $gf$ is a deflation, then so is $g$.
\end{cond}

\begin{Ex}\label{Ex:4.12}

\emph{(1)} If $\mathcal{C}$ is an exact category, then Condition \emph{(WIC)} is equivalent to $\mathcal{C}$ is
weakly idempotent complete \emph{(see \cite[Proposition 7.6]{B"u})}.

\emph{(2)} If $\mathcal{C}$ is a triangulated category, then Condition \emph{(WIC)} is automatically satisfied.
\end{Ex}

   \begin{prop}\label{pro1}
   \emph{(1)} Let $f\in\mathcal{C}(A, B), g\in\mathcal{C}(B, C)$ be any composable pair of morphisms. If $gf$ is a $\xi$-inflation, then so is $f$.

   \emph{(2)} Dual of \emph{(1)}. \end{prop}

   \begin{proof} (1) If $gf$ is a $\xi$-inflation, then  $f$ is an inflation by Condition (WIC). Assume that \\ $\xymatrix{A\ar[r]^f&B\ar[r]^{f'}&C\ar@{-->}[r]^{\delta}&}$ is an $\mathbb{E}$-triangle and $\xymatrix{A\ar[r]^{gf}&B'\ar[r]^{h}&C'\ar@{-->}[r]^{\delta'}&}$ is an $\mathbb{E}$-triangle in $\xi$. Applying (ET3) to
     $$\xymatrix{A\ar[r]^{f}\ar@{=}[d]&B\ar[r]^{f'}\ar[d]^{g}&C\ar@{-->}[r]^{\delta}&\\
   A\ar[r]^{gf}&B'\ar[r]^{h}&C'\ar@{-->}[r]^{\delta'}&}$$
   we obtain a morphism $c\in \mathcal{C}(C, C')$ satisfying $\delta=c^*\delta$, hence $\xymatrix@C=2em{A\ar[r]^{f}&B\ar[r]^{f'}&C\ar@{-->}[r]^{\delta}&}$ is  an $\mathbb{E}$-triangle in $\xi$.
   \end{proof}
\begin{lem}\label{lem3}  Let
$$
\xymatrix{
  K \ar[d]_{k}  & K' \ar[d]_{k'}  & K'' \ar[d]_{k''}  & \\
  A\ar@{}[dr]|{\circlearrowleft} \ar[d]_{a} \ar[r]^{x} & B \ar@{}[dr]|{\circlearrowleft}\ar[d]_{b} \ar[r]^{y} & C \ar[d]_{c'} \ar@{-->}[r]^{\delta} &  \\
  A' \ar@{-->}[d]_{\kappa} \ar[r]_{x'} & B' \ar@{-->}[d]_{\kappa'} \ar[r]^{y'} & C' \ar@{-->}[d]_{\kappa''} \ar@{-->}[r]^{\delta'} &  \\
  &  & &  }
  $$
 be a diagram  of $\mathbb{E}$-triangles.  If $y$ is epic, then there is an $\mathbb{E}$-triangle $\xymatrix{K\ar[r]^m&K'\ar[r]^{n'}&K''\ar@{-->}[r]^{\delta''}&}$
 which { makes} the following diagram commutative
 $$
\xymatrix{
  K\ar@{}[dr]|{\circlearrowleft} \ar[d]_{k}\ar[r]^m  & K' \ar@{}[dr]|{\circlearrowleft}\ar[d]_{k'}\ar[r]^{n'}  & K'' \ar[d]_{k''}\ar@{-->}[r]^{\delta''}  & \\
  A\ar@{}[dr]|{\circlearrowleft} \ar[d]_{a} \ar[r]^{x} & B \ar@{}[dr]|{\circlearrowleft}\ar[d]_{b} \ar[r]^{y} & C \ar[d]_{c'} \ar@{-->}[r]^{\delta} &  \\
  A' \ar@{-->}[d]_{\kappa} \ar[r]^{x'} & B' \ar@{-->}[d]_{\kappa'} \ar[r]^{y'} & C' \ar@{-->}[d]_{\kappa''} \ar@{-->}[r]^{\delta'} &  \\
  &  & &  }
  $$
{ where} $(k, k', k''), (m, x, x')$ and $(n', y, y')$ are morphisms of $\mathbb{E}$-triangles.
\end{lem}
  \begin{proof} It follows from \cite[Lemma 5.9]{NP} that there exist $\mathbb{E}$-triangles $\xymatrix{K\ar[r]^m&K'\ar[r]^{n}&X\ar@{-->}[r]^{\nu}&}$
  and $\xymatrix{X\ar[r]^i&C\ar[r]^{c}&C'\ar@{-->}[r]^{\tau}&}$ which make the following diagram commutative
 $$ \xymatrix{
  K \ar@{}[dr]|{\circlearrowleft}\ar[d]_{k}\ar[r]^m  & K'\ar@{}[dr]|{\circlearrowleft} \ar[d]_{k'}\ar[r]^{n}  & X \ar[d]_{i}\ar@{-->}[r]^{\nu}  & \\
  A \ar@{}[dr]|{\circlearrowleft}\ar[d]_{a} \ar[r]^{x} & B\ar@{}[dr]|{\circlearrowleft} \ar[d]_{b} \ar[r]^{y} & C \ar[d]_{c} \ar@{-->}[r]^{\delta} &  \\
  A' \ar@{-->}[d]_{\kappa} \ar[r]^{x'} & B' \ar@{-->}[d]_{\kappa'} \ar[r]^{y'} & C' \ar@{-->}[d]_{\tau} \ar@{-->}[r]^{\delta'} &  \\
  &  & &  }
  $$
{ where} $(k, k', i), (a, b, c), (m, x, x')$ and $(n, y, y')$ are morphisms of $\mathbb{E}$-triangles.
Since $c'y=y'b=cy$ and $y$ is epic, $c'=c$. By (ET3)$^{op}$ and \cite[Corollary 3.6]{NP}, we obtain a morphism of $\mathbb{E}$-triangles
$$\xymatrix{K''\ar[r]^{k''}\ar[d]^u&C\ar@{=}[d]\ar[r]^{c'}&C'\ar@{=}[d]\ar@{-->}[r]^{\kappa''}&\\
X\ar[r]^i&C\ar[r]^{c}&C'\ar@{-->}[r]^{\tau}&\\
}$$
where $u$ is an isomorphism. It is clear that $u_*\kappa''=\tau$. Set $n'=u^{-1}n$, then $$\xymatrix{K\ar[r]^m&K'\ar[r]^{n'}&K''\ar@{-->}[r]^{\delta''}&}$$ is an
$\mathbb{E}$-triangle by \cite[Proposition 3.7]{NP} where $\delta''=u^*\nu$.
It is easy to check that $k''n'=yk'$. Moreover, one can show that $k_*\delta''=k_*u^*\nu=u^*k_*\nu=u^*i^*\delta=(iu)^*\delta=k^{''*}\delta$ and $n'_*\kappa'=u^{-1}_*n_*\kappa'=u^{-1}_*y^{'*}\tau=y^{'*}u^{-1}\tau=y^{'*}\kappa''$, which implies  that $(k, k', k'')$ and $(n', y, y')$ are morphisms of $\mathbb{E}$-triangles.
  \end{proof}

\begin{lem}\label{leme} Let $\xymatrix@C=2em{A\ar[r]^x& B\ar[r]^y& C\ar@{-->}[r]^\delta&}$ and $\xymatrix@C=2em{A'\ar[r]^{x'}& B'\ar[r]^{y'}& C'\ar@{-->}[r]^{\delta'}&}$ be $\mathbb{E}$-triangles. Then the following hold.

\emph{(1)} If $(a, c): \delta\rightarrow \delta'$ is a morphism of $\mathbb{E}$-triangles where $a, c$ are inflations, then there is an inflation $b: B\rightarrow B'$ which { makes} the following diagram commutative.
$$\xymatrix{A\ar[r]^x\ar[d]_a&B\ar[r]^y\ar[d]^b&C\ar[d]^c\ar@{-->}[r]^{\delta}&\\
A'\ar[r]^{x'}& B'\ar[r]^{y'}& C'\ar@{-->}[r]^{\delta'}&}$$

\emph{(2)} Dual of \emph{(1)}
\end{lem}
\begin{proof} (1) It follows from Lemma \ref{lem1}(1), and (ET4)$^{\rm op}$, there exist inflation morphisms $d, h, d' $ which make the following diagram commutative
$$\xymatrix{A\ar[r]^x\ar[d]_a&B\ar[r]^y\ar[d]^d&C\ar@{=}[d]\ar@{-->}[r]^{\delta}&\\
A'\ar[r]^f\ar@{=}[d]& D\ar[r]^g\ar[d]^h& C\ar@{=}[d]\ar@{-->}[r]^{a_*\delta}&\\
A'\ar[r]^{f'}\ar@{=}[d]&D'\ar[r]^{g'}\ar[d]^{d'}&C\ar[d]^c\ar@{-->}[r]^{c^*\delta'}&\\
A'\ar[r]^{x'}&{ B'}\ar[r]^{y'}&C'\ar@{-->}[r]^{\delta'}&}$$
where $h$ is an isomorphism by \cite[Corollary 3.6(1)]{NP}. Set $b=d'hd$, then it is an inflation by Corollary \ref{cor1}.
\end{proof}

\begin{thm}\label{thm1}  If $\xymatrix@C=2em{A\ar[r]^x& B\ar[r]^y& C\ar@{-->}[r]^\rho&}$ is an $\mathbb{E}$-triangle in $\xi$ with $C\in\mathcal{GP}(\xi)$, then  $A\in\mathcal{GP}(\xi)$ if and only if $B\in\mathcal{GP}(\xi)$.
\end{thm}
\begin{proof} If $A\in\mathcal{GP}(\xi)$, then there exists a $\mathcal{C}(-,\mathcal{P}(\xi))$-exact $\mathbb{E}$-triangle $\xymatrix@C=2em{A\ar[r]^{g_{-1}^A}& P_{-1}^A\ar[r]^{f_{-1}^A}& K_{-1}^A\ar@{-->}[r]^{\delta_{-1}^A}&}$ in $\xi$ with $P_{-1}^A\in\mathcal{P}(\xi)$ and $K_{-1}^A\in\mathcal{GP}(\xi)$. Similarly, there exists a $\mathcal{C}(-,\mathcal{P}(\xi))$-exact $\mathbb{E}$-triangle $\xymatrix@C=2em{C\ar[r]^{g_{-1}^C}& P_{-1}^C\ar[r]^{f_{-1}^C}& K_{-1}^C\ar@{-->}[r]^{\delta_{-1}^C}&}$ in $\xi$ with $P_{-1}^C\in\mathcal{P}(\xi)$ and $K_{-1}^C\in\mathcal{GP}(\xi)$. It follows from Lemma \ref{lem5}(2) that the $\mathbb{E}$-triangle $\xymatrix@C=2em{A\ar[r]^x& B\ar[r]^y& C\ar@{-->}[r]^\delta&}$ is $\mathcal{C}(-,\mathcal{P}(\xi))$-exact. It is easy to check that $(g_{-1}^A)_*\rho=0=(g_{-1}^C)^*0$, hence there is an inflation { $g_{-1}^B: B\rightarrow P_{-1}^A\oplus P_{-1}^C=:P_{-1}^B$} which makes the following  diagram commutative
$$\xymatrix{A\ar[r]^x\ar[d]^{g_{-1}^A}&B\ar[r]^y\ar[d]^{g_{-1}^B}&C\ar[d]^{g_{-1}^C}\ar@{-->}[r]^{\rho}&\\
P_{-1}^A\ar[r]^{\tiny\begin{bmatrix}1\\0\end{bmatrix}}&P_{-1}^B\ar[r]^{\tiny\begin{bmatrix}0&1\end{bmatrix}}
&P_{-1}^C\ar@{-->}[r]^{0}&\\
}$$
by  Lemma \ref{leme}(1). Hence there exists an $\mathbb{E}$-triangle $\xymatrix@C=2em{B\ar[r]^{g_{-1}^B}& P_{-1}^B\ar[r]^{f_{-1}^B}& K_{-1}^B\ar@{-->}[r]^{\delta_{-1}^B}&}$. By the dual of Lemma \ref{lem3}, there exists a commutative diagram
$$\xymatrix{&A\ar[r]^x\ar[d]^{g_{-1}^A}&B\ar[r]^y\ar[d]^{g_{-1}^B}&C\ar[d]^{g_{-1}^C}\ar@{-->}[r]^{\rho}&\\
{(\dag)}&P_{-1}^A\ar[r]^{\tiny\begin{bmatrix}1\\0\end{bmatrix}}\ar[d]^{f_{-1}^A}&P_{-1}^B\ar[d]^{f_{-1}^B}\ar[r]^{\tiny\begin{bmatrix}0&1\end{bmatrix}}
&P_{-1}^C\ar[d]^{f_{-1}^C}\ar@{-->}[r]^{0}&\\
&K_{-1}^A\ar@{-->}[d]^{\delta_{-1}^A}\ar[r]^{x_{-1}}&K_{-1}^B\ar[r]^{y_{-1}}\ar@{-->}[d]^{\delta_{-1}^B}&K_{-1}^C
\ar@{-->}[d]^{\delta_{-1}^C}\ar@{-->}[r]^{\rho_{-1}}&\\
&&&&}$$
made of $\mathbb{E}$-triangles. Since $f_{-1}^C$ and ${\tiny \begin{bmatrix}0&1\end{bmatrix}}$ are $\xi$-deflations, so is $y_{-1}$ by Proposition \ref{pro1}(2). It is easy to check that  the $\mathbb{E}$-triangle $\xymatrix{K_{-1}^A\ar[r]^{x_{-1}}&K_{-1}^B\ar[r]^{y_{-1}}&K_{-1}^C
\ar@{-->}[r]^{\rho_{-1}}&}$ is isomorphic to an $\mathbb{E}$-triangle in $\xi$ by \cite[Corollary 3.6(3)]{NP}, hence it is an $\mathbb{E}$-triangle in $\xi$. Applying $\mathcal{C}(\mathcal{P}(\xi),-)$ to the above commutative diagram, it is easy to prove that the  $\mathbb{E}$-triangle $\xymatrix@C=2em{B\ar[r]^{g_{-1}^B}& P_{-1}^B\ar[r]^{f_{-1}^B}& K_{-1}^B\ar@{-->}[r]^{\delta_{-1}^B}&}$ is $\mathcal{C}(\mathcal{P}(\xi),-)$-exact by a diagram chasing, hence it is an $\mathbb{E}$-triangle in $\xi$ by Lemma \ref{lem4}. Applying $\mathcal{C}(-,\mathcal{P}(\xi))$ to the above commutative diagram, it is easy to { show} that the   $\mathbb{E}$-triangle $\xymatrix@C=2em{B\ar[r]^{g_{-1}^B}& P_{-1}^B\ar[r]^{f_{-1}^B}& K_{-1}^B\ar@{-->}[r]^{\delta_{-1}^B}&}$  is $\mathcal{C}(-,\mathcal{P}(\xi))$-exact by a diagram chasing. { Proceeding} in this manner, one can get a $\mathcal{C}(-,\mathcal{P}(\xi))$-exact $\xi$-exact complex $\xymatrix{B\ar[r]&P_{-1}^B\ar[r]&P_{-2}^B\ar[r]&\cdots}$. Dually, we can obtain a $\mathcal{C}(-,\mathcal{P}(\xi))$-exact $\xi$-exact complex $\xymatrix{\cdots\ar[r]&P_{1}^B\ar[r]&P_{0}^B\ar[r]&B.}$ Hence there is a complete $\xi$-projective resolution
$\xymatrix{\cdots\ar[r]&P_{1}^B\ar[r]&P_{0}^B\ar[r]&P_{-1}^B\ar[r]&P_{-2}^B\ar[r]&\cdots}$ by pasting  these $\xi$-exact complexes together, which implies that $B\in\mathcal{GP}(\xi)$.

Assume that $B\in\mathcal{GP}(\xi)$, then there is a $\mathcal{C}(-,\mathcal{P}(\xi))$-exact $\mathbb{E}$-triangle $$\xymatrix@C=2em{B\ar[r]^{g_{-1}^B}& P_{-1}^B\ar[r]^{f_{-1}^B}& K_{-1}^B\ar@{-->}[r]^{\delta_{-1}^B}&}$$ in $\xi$ with $P_{-1}^B\in\mathcal{P}(\xi)$ and $K_{-1}^B\in\mathcal{GP}(\xi)$. Hence there exists the following commutative diagram
$$\xymatrix{A\ar[r]^x\ar@{=}[d]&B\ar[r]^y\ar[d]^{g_{-1}^B}&C\ar[d]^{g}\ar@{-->}[r]^{\rho}&\\
A\ar[r]^{g_{-1}^A}&P_{-1}^B\ar[d]^{f_{-1}^B}\ar[r]^{f_{-1}^A}
&G\ar[d]^{f}\ar@{-->}[r]^{\delta_{-1}^A}&\\
&K_{-1}^B\ar@{=}[r]\ar@{-->}[d]^{\delta_{-1}^B}&K_{-1}^B\ar@{-->}[d]^{y_*\delta_{-1}^B}
&\\
&&&}$$
made of $\mathbb{E}$-triangles. It is easy to check that these $\mathbb{E}$-triangles are $\mathcal{C}(\mathcal{P}(\xi), -)$-exact, hence they are $\mathbb{E}$-triangles in $\xi$.
{ Since $\xymatrix{C\ar[r]^{g_{-1}^A}&G\ar[r]^{f_{-1}^A}
&K^B_{-1}\ar@{-->}[r]&}$ is an $\mathbb{E}$-triangle in $\xi$ with
$C,K^B_{-1}\in\mathcal{GP}(\xi)$ and $\mathcal{GP}(\xi)$ is closed under extensions by the above proof, }
we can conclude that $G\in\mathcal{GP}(\xi)$, which implies that the $\mathbb{E}$-triangle
$\xymatrix{A\ar[r]^{g_{-1}^A}&P_{-1}^B\ar[r]^{f_{-1}^A}
&G\ar@{-->}[r]^{\delta_{-1}^A}&}$ is $\mathcal{C}(-,\mathcal{P}(\xi))$-exact by Lemma \ref{lem5}(2).
Note that $G\in\mathcal{GP}(\xi)$, there is a $\mathcal{C}(-,\mathcal{P}(\xi))$-exact $\xi$-exact complex $$\xymatrix{G\ar[r]&P_{-2}^A\ar[r]&P_{-3}^A\ar[r]&\cdots}$$ with $P_{-i}^A\in\mathcal{P}(\xi)$ for any $i\geqslant 2$. Hence we get a $\mathcal{C}(-,\mathcal{P}(\xi))$-exact $\xi$-exact complex $$\xymatrix{A\ar[r]&P_{-1}^B\ar[r]&P_{-2}^A\ar[r]&\cdots}$$ with $P_{-1}^B\in\mathcal{P}(\xi)$ and $P_{-i}^A\in\mathcal{P}(\xi)$ for any $i\geqslant 2$. It { now suffices} to show that there exists a $\mathcal{C}(-,\mathcal{P}(\xi))$-exact $\xi$-exact complex $\xymatrix{\cdots\ar[r]&P_1^A\ar[r]&P_{0}^A\ar[r]&A}$ with $P_{i}^A\in\mathcal{P}(\xi)$  for any $i\geqslant 0$. Assume that $\xymatrix{K_1^A\ar[r]^{g_0^A}&P_0^A\ar[r]^{f_0^A}&A\ar@{-->}[r]^{\delta_0^A}&}$ and $\xymatrix{K_1^C\ar[r]^{g_0^C}&P_0^C\ar[r]^{f_0^C}&C\ar@{-->}[r]^{\delta_0^C}&}$ are $\mathbb{E}$-triangles in $\xi$ with $P_0^A, P_0^C\in\mathcal{P}(\xi)$ and $K_1^C\in\mathcal{GP}(\xi)$. { Following an argument similar to the proof that $\mathcal{GP}(\xi)$ is closed under extensions, one can show that} there exists a commutative diagram
$$\xymatrix{K_{1}^A\ar[d]^{g_0^A}\ar[r]^{x_{1}}&K_{1}^B\ar[d]^{g_0^B}\ar[r]^{y_{1}}&K_{1}^C\ar[d]^{g_0^C}
\ar@{-->}[r]^{\rho_{1}}&\\
P_{0}^A\ar[r]^{\tiny\begin{bmatrix}1\\0\end{bmatrix}}\ar[d]^{f_{0}^A}&P_{0}^B\ar[d]^{f_{0}^B}\ar[r]^{\tiny\begin{bmatrix}0&1\end{bmatrix}}
&P_{0}^C\ar[d]^{f_{0}^C}\ar@{-->}[r]^{0}&\\
A\ar[r]^x\ar@{-->}[d]^{\delta_{0}^A}&B\ar[r]^y\ar@{-->}[d]^{\delta_{0}^B}&C\ar@{-->}[d]^{\delta_{0}^C}\ar@{-->}[r]^{\rho}&\\
&&&}$$
made of $\mathbb{E}$-triangles in $\xi$. Note that the $\mathbb{E}$-triangles except the first vertical in the above diagram are $\mathcal{C}(-,\mathcal{P}(\xi))$-exact by Lemma \ref{lem5}(2), it is easy to show that the first vertical is $\mathcal{C}(-,\mathcal{P}(\xi))$-exact.
Since any { syzygy} of $B$ is a direct summand of { a} $\xi$-$\mathcal{G}$projective object, any $\mathbb{E}$-triangle $\xymatrix{K_2^B\ar[r]^{g_1^B}&P_1^B\ar[r]^{f_1^B}&K_1^B\ar@{-->}[r]^{\delta_1^B}&}$ in $\xi$ is $\mathcal{C}(-,\mathcal{P}(\xi))$-exact by Lemma \ref{lem5}(3). { Proceeding} in this manner, one can get a $\mathcal{C}(-,\mathcal{P}(\xi))$-exact $\xi$-exact complex $\xymatrix{\cdots\ar[r]&P_1^A\ar[r]&P_{0}^A\ar[r]&A}$ with $P_{i}^A\in\mathcal{P}(\xi)$  for any $i\geqslant 0$. Hence $A\in\mathcal{GP}(\xi)$, as desired.
 \end{proof}

\begin{thm}\label{thm2}  $\mathcal{GP}(\xi)$ is closed under direct { summands}.
\end{thm}
\begin{proof} Assume that $G\in\mathcal{GP}(\xi)$ and $H$ is a direct summand of $G$, then there exists $H'\in\mathcal{C}$ such that $G=H\oplus H'$. Therefore there exist two split $\mathbb{E}$-triangles  $\xymatrix@C=2em{H\ar[r]^{\tiny\begin{bmatrix}1\\0\end{bmatrix}}&G\ar[r]^{\tiny\begin{bmatrix}0&1\end{bmatrix}\ \ }&H'\ar@{-->}[r]^0&}$ and $\xymatrix@C=2em{H'\ar[r]^{\tiny\begin{bmatrix}0\\1\end{bmatrix}}&G\ar[r]^{\tiny\begin{bmatrix}1&0\end{bmatrix}\ \ }&H\ar@{-->}[r]^0&}$
in $\xi$. Since $G\in\mathcal{GP}(\xi)$, there is an $\mathbb{E}$-triangle $$\xymatrix{G\ar[r]^{\alpha_{-1}}&P_{-1}\ar[r]^{\beta_{-1}}&K_{-1}\ar@{-->}[r]^{\delta_{-1}}&}$$ in $\xi$ with $P_{-1}\in\mathcal{P}(\xi)$ and $K_{-1}\in\mathcal{GP}(\xi)$. Hence there exists a commutative diagram

$$\xymatrix@C=3em{
  H\ar[r]^{\tiny\begin{bmatrix}1\\0\end{bmatrix}} \ar@{=}[d] &G \ar[r]^{\tiny\begin{bmatrix}0&1\end{bmatrix}} \ar[d]^{\alpha_{-1}}& H' \ar[d]^{\alpha'_{-1}} \ar@{-->}[r]^0 &  \\
  H\ar[r]^{g_{-1}} & P_{-1} \ar[r]^{f_{-1}} \ar[d]^{\beta_{-1}} & X\ar@{-->}[r]^{\rho_{-1}} \ar[d]^{\beta'_{-1}} &\\
  & K_{-1} \ar@{-->}[d]^{\delta_{-1}} \ar@{=}[r] &K_{-1}\ar@{-->}[d]^{\tiny\begin{bmatrix}0&1\end{bmatrix}_*\delta_{-1}}& \\
  &&&}
$$
made of  $\mathbb{E}$-triangles in $\xi$ because it is easy to check that these $\mathbb{E}$-triangles are $\mathcal{C}(\mathcal{P}(\xi), -)$-exact. Since the first row, the second and the third columns are $\mathcal{C}(-,\mathcal{P}(\xi))$-exact by Lemma \ref{lem5}(2), it is easy to show that $\xymatrix@C=2em{H\ar[r]^{g_{-1}}&P_{-1}\ar[r]^{f_{-1}}&X\ar@{-->}[r]^{\rho_{-1}}&}$ is $\mathcal{C}(-,\mathcal{P}(\xi))$-exact by diagram chasing.
Note that $\xymatrix@C=2em{H'\ar[r]^{\alpha'_{-1}}&X\ar[r]^{\beta'_{-1}}&K_{-1}\ar@{-->}[r]^{\tiny\begin{bmatrix}0&1\end{bmatrix}_*\delta_{-1}}&}$  and $\xymatrix@C=2em{H'\ar[r]^{\tiny\begin{bmatrix}0\\1\end{bmatrix}}&G\ar[r]^{\tiny\begin{bmatrix}1&0\end{bmatrix}\ \ }&H\ar@{-->}[r]^0&}$ are  $\mathbb{E}$-triangles in $\xi$. { Then} there exists a commutative diagram
$$\xymatrix@C=3em{
    H' \ar[d]_{\tiny\begin{bmatrix}0\\1\end{bmatrix}} \ar[r]^{\alpha'_{-1}} & X \ar[d]_{g'_{-1}} \ar[r]^{\beta'_{-1}} & K_{-1} \ar@{=}[d] \ar@{-->}[r]^{\tiny\begin{bmatrix}0&1\end{bmatrix}_*\delta_{-1}}&\\
  G \ar[d]_{\tiny\begin{bmatrix}1&0\end{bmatrix}} \ar[r]^{\alpha''_{-1}} & G_{-1} \ar[d]_{f'_{-1}} \ar[r]^{\beta''_{-1}}&K_{-1}\ar@{-->}[r]^{\delta_{-1}'}&\\
    H \ar@{=}[r]\ar@{-->}[d]^0&  H \ar@{-->}[d]^0& & \\
    && &}
$$
such that $\xymatrix@C=2em{G\ar[r]^{\alpha''_{-1}}&G_{-1}\ar[r]^{\beta''_{-1}}&K_{-1}\ar@{-->}[r]^{\delta_{-1}'}&}$ and $\xymatrix@C=2em{X\ar[r]^{g'_{-1}}&G_{-1}\ar[r]^{f'_{-1}}&H\ar@{-->}[r]^{0}&}$ are $\mathbb{E}$-triangles in $\xi$  because $\xi$ is closed under cobase change. It follows from Theorem \ref{thm1} that  $G_{-1}\in\mathcal{GP}(\xi)$ as $G, K_{-1}\in\mathcal{GP}(\xi)$, so there exists an $\mathbb{E}$-triangle
$\xymatrix@C=2em{G_{-1}\ar[r]^{\alpha_{-2}}&P_{-2}\ar[r]^{\beta_{-2}}&K_{-2}\ar@{-->}[r]^{\delta_{-2}}&}$ in $\xi$ with $P_{-2}\in\mathcal{P}(\xi)$ and $K_{-2}\in\mathcal{GP}(\xi)$. Hence there exists a commutative diagram
$$\xymatrix@C=3em{
  X\ar[r]^{g_{-1}'} \ar@{=}[d] &G_{-1} \ar[r]^{f_{-1}'} \ar[d]^{\alpha_{-2}}& H \ar[d]^{\alpha'_{-1}} \ar@{-->}[r]^0 &  \\
  X\ar[r]^{g_{-2}} & P_{-2} \ar[r]^{f_{-2}} \ar[d]^{\beta_{-2}} & Y\ar@{-->}[r]^{\rho_{-2}} \ar[d]^{\beta'_{-2}} &\\
  & K_{-2} \ar@{-->}[d]^{\delta_{-2}} \ar@{=}[r] &K_{-2}\ar@{-->}[d]^{(f'_{-1})_*\delta_{-2}}& \\
  &&&}
$$ made of  $\mathbb{E}$-triangles in $\xi$ because it is easy to check that these $\mathbb{E}$-triangles are $\mathcal{C}(\mathcal{P}(\xi), -)$-exact. Similar to the proof above, one can prove that  the $\mathbb{E}$-triangle $\xymatrix@C=2em{X\ar[r]^{g_{-2}}&P_{-2}\ar[r]^{f_{-2}}&Y\ar@{-->}[r]^{\rho_{-2}}&}$ is $\mathcal{C}(-,\mathcal{P}(\xi))$-exact.  One can get a $\mathcal{C}(-,\mathcal{P}(\xi))$-exact $\xi$-exact complex $\xymatrix@C=1.7em{H\ar[r]&P_{-1}\ar[r]& P_{-2}\ar[r]&\cdots}$ with each $P_{-i}\in\mathcal{P}(\xi)$ for $i\geqslant 1$ by proceeding in this manner. Dually, we can obtain a $\mathcal{C}(-,\mathcal{P}(\xi))$-exact $\xi$-exact complex $\xymatrix@C=2em{\cdots\ar[r]&P_{1}\ar[r]& P_{0}\ar[r]&H}$ with each $P_{i}\in\mathcal{P}(\xi)$ for $i\geqslant 0$. Hence $H\in\mathcal{GP}(\xi)$, as desired.
\end{proof}

As a consequence of Example \ref{Ex:4.12}(2) and Theorem \ref{thm2}, we have the following corollary.
\begin{cor}\label{corllary3} Let $\mathcal{C}$ be a triangulated category and $\xi$ a proper class of triangles. Then $\mathcal{GP}(\xi)$ is closed under direct { summands}.
\end{cor}
\begin{rem} \label{rem:4.19} We note that Corollary \ref{corllary3} refines a result of Asadollahi and Salarian in \cite{AS1}. In their paper, they showed that for any triangulated cateogory $\mathcal{C}$ with enough $\mathcal{E}$-projectives, if the class $\mathcal{GP}(\mathcal{E})$ of $\mathcal{E}$-$\mathcal{G}$projective objects in $\mathcal{C}$ is closed under countable direct sums, then $\mathcal{GP}(\mathcal{E})$ is closed under direct summands, see the proof of  \cite[Proposition 3.13]{AS1}.
\end{rem}

\begin{cor}\label{pro3} If $\xymatrix@C=1.5em{A\ar[r]&B\ar[r]&C\ar@{-->}[r]&}$ is an $\mathbb{E}$-triangle in $\xi$ with  $A,B\in \mathcal{GP}(\xi)$, then $C\in\mathcal{GP}(\xi)$ if and only if any $\mathbb{E}$-triangle$\xymatrix@C=1.5em{A'\ar[r]&B'\ar[r]&C\ar@{-->}[r]&}$ in $\xi$ is $\mathcal{C}(-,\mathcal{P}(\xi))$-exact.
\end{cor}
\begin{proof}  The ``only if" part is clear. For the ``if" part, since $A\in \mathcal{GP}(\xi)$, there exists an $\mathbb{E}$-triangle
$\xymatrix@C=1.5em{A\ar[r]&P\ar[r]&K\ar@{-->}[r]&}$ in $\xi$ with $P\in \mathcal{P}(\xi)$ and $K\in \mathcal{GP}(\xi)$.
Then we have the following commutative diagram
$$\xymatrix{A\ar[r]\ar[d]&B\ar[r]\ar[d]&C\ar@{=}[d]\ar@{-->}[r]&\\
 P\ar[r]^x\ar[d]&G\ar[r]^y\ar[d]&C\ar@{-->}[r]&\\
 K\ar@{-->}[d]\ar@{=}[r]&K\ar@{-->}[d]&&\\
 &&&}$$
where all rows and columns are $\mathbb{E}$-triangles in $\xi$ because $\xi$ is closed under cobase change. { It follows that $G\in \mathcal{GP}(\xi)$
since $B,K\in \mathcal{GP}(\xi)$ and $\mathcal{GP}(\xi)$ is closed under extensions by Theorem \ref{thm1}.
We know that the $\mathbb{E}$-triangle $\xymatrix@C=2em{P\ar[r]^x&G\ar[r]^y&C\ar@{-->}[r]&}$ in $\xi$ is
$\C(-,\mathcal P(\xi))$-exact by hypothesis.  Since $P\in\mathcal P(\xi)$, we have that
$x$ is a section. Thus $\xymatrix@C=2em{P\ar[r]^x&G\ar[r]^y&C\ar@{-->}[r]&}$ splits. }  Hence $C\in\mathcal{GP}(\xi)$ by Theorem \ref{thm2}.
\end{proof}
\begin{rem} Using arguments analogous to the ones employed in this subsection, one may { show} that if  $(\mathcal{C}, \mathbb{E}, \mathfrak{s})$ is an extriangulated category satisfying Condition \emph{(WIC)}, then the category of $\xi$-$\mathcal{G}$injective objects is full, additive, closed under isomorphisms { and} direct summands. The details are left to the reader.
\end{rem}

\section{\bf $\xi$-$\mathcal{G}$projective dimension and addmissible model structure}

Throughout this section, { we always assume that $(\mathcal{C}, \mathbb{E}, \mathfrak{s})$ is
an extriangulated category and $\xi$ a proper class of $\mathbb{E}$-triangles. Moreover $\C$ has enough $\xi$-projectives and enough $\xi$-injectives, and satisfies Condition (WIC).}

\subsection{$\xi$-$\mathcal{G}$projective dimension}
The $\xi$-$\mathcal{G}$projective dimension $\xi$-$\mathcal{G}$${\rm pd} A$ of an object $A\in\mathcal{C}$ is defined inductively.
 If $A\in\mathcal{GP}(\xi)$ then define $\xi$-$\mathcal{G}$${\rm pd} A=0$.
Next by induction, for an integer $n>0$, put $\xi$-$\mathcal{G}$${\rm pd} A\leqslant n$ if there exists an $\mathbb{E}$-triangle $\xymatrix{K\ar[r]&G\ar[r]&A\ar@{-->}[r]&}$ in $\xi$ with $G\in \mathcal{GP}(\xi)$ and $\xi$-$\mathcal{G}{\rm pd} K\leqslant n-1$.

We define $\xi$-$\mathcal{G}$${\rm pd} A=n$ if $\xi$-$\mathcal{G}{\rm pd} A\leqslant n$ and $\xi$-$\mathcal{G}{\rm pd} A\nleq n-1$. If $\xi$-$\mathcal{G}$${\rm pd} A\neq n$ for all $n\geqslant 0$, we set $\xi$-$\mathcal{G}$${\rm pd} A=\infty$.


We let $\widetilde{\mathcal{GP}}(\xi)$ (respectively $\widetilde{\mathcal{P}}(\xi)$)  denote the full subcategory of $\mathcal{C}$ whose objects are of finite $\xi$-$\mathcal{G}$projective (respectively $\xi$-projective) dimension. The following observation is useful in this section.

\begin{lem}\label{pro4} If $\xymatrix{X\ar[r]&Y\ar[r]&Z\ar@{-->}[r]&}$ is an $\mathbb{E}$-triangle in $\xi$ such that $X\neq 0$ and $Z\in\mathcal{GP}(\xi)$, { then} $\xi$-$\mathcal{G}{\rm pd} X=\xi$-$\mathcal{G}{\rm pd} Y$.
\end{lem}
 \begin{proof} The result is clear from Theorem \ref{thm1} if one of the $X$ or $Y$ is $\xi$-$\mathcal{G}$projective. Let $\xi$-$\mathcal{G}{\rm pd} X=n>0$. There exists an $\mathbb{E}$-triangle $\xymatrix{K_X\ar[r]&G_X\ar[r]&X\ar@{-->}[r]&}$ in $\xi$ where $G_X\in\mathcal{GP}(\xi)$ and $\xi$-$\mathcal{G}{\rm pd} K_X=n-1$. Since $Z\in\mathcal{GP}(\xi)$, there exists an $\mathbb{E}$-triangle  $\xymatrix{K_Z\ar[r]&P_Z\ar[r]&Z\ar@{-->}[r]&}$ in $\xi$ where $P_Z\in\mathcal{P}(\xi)$ and $K_Z\in\mathcal{GP}(\xi)$. { Dual to the proof that we obtain the commutative diagram $(\dag)$ in Theorem \ref{thm1}}, we have the following commutative diagram
$$\xymatrix{K_X\ar[r]\ar[d]&K_Y\ar[r]\ar[d]&K_Z\ar[d]\ar@{-->}[r]&\\
G_X\ar[d]\ar[r]&G_X\oplus P_Z\ar[d]\ar[r]&
P_Z\ar[d]\ar@{-->}[r]&\\
X\ar[r]\ar@{-->}[d]&Y\ar@{-->}[d]\ar[r]&Z\ar@{-->}[d]\ar@{-->}[r]&\\
&&&
}$$
 { where} all rows and columns are $\mathbb{E}$-triangles in $\xi$. Hence $G_X\oplus P_Z\in\mathcal{GP}(\xi)$ because $G_X$ and $P_Z$ are $\xi$-$\mathcal{G}$projective. Now consider the $\mathbb{E}$-triangle $\xymatrix{K_X\ar[r]&K_Y\ar[r]&K_Z\ar@{-->}[r]&}$ in $\xi$, where $K_Z\in\mathcal{GP}(\xi)$ and $\xi$-$\mathcal{G}{\rm pd} K_X=n-1$. We can deduce that $\xi$-$\mathcal{G}{\rm pd} K_Y=n-1$ by induction, and hence $\xi$-$\mathcal{G}{\rm pd} Y=n$.

Now suppose $\xi$-$\mathcal{G}{\rm pd} Y=n>0$. So there exists an $\mathbb{E}$-triangle $\xymatrix{K_Y\ar[r]&G_Y\ar[r]&Y\ar@{-->}[r]&}$ in $\xi$ where $G_Y\in\mathcal{GP}(\xi)$ and $\xi$-$\mathcal{G}{\rm pd} K_Y=n-1$. Consider the following commutative diagram
$$\xymatrix{K_Y\ar@{=}[r]\ar[d]&K_Y\ar[d]&&\\
G_X\ar[r]\ar[d]&G_Y\ar[r]\ar[d]&Z\ar@{=}[d]\ar@{-->}[r]&\\
X\ar[r]\ar@{-->}[d]&Y\ar[r]\ar@{-->}[d]&Z\ar@{-->}[r]&\\
&&&}$$
 { where} all rows and columns are $\mathbb{E}$-triangles  by (ET4)$^{\rm op}$. It is easy to check that these $\mathbb{E}$-triangles are $\mathcal{C}(\mathcal{P}(\xi),-)$-exact by Snake Lemma, hence these $\mathbb{E}$-triangles are in $\xi$ by Lemma \ref{lem4}. So $G_X\in\mathcal{GP}(\xi)$ by Theorem \ref{thm1} because $Z$ and $G_Y$ are $\xi$-$\mathcal{G}$projective. Hence $\xi$-$\mathcal{G}{\rm pd} X=n$. By induction now  we complete the proof.
\end{proof}

\begin{prop}\label{pro5} { The following are equivalent for any object $A$ in $(\mathcal{C}, \mathbb{E}, \mathfrak{s})$}:

\emph{(1)} $\xi$-$\mathcal{G}{\rm pd} A\leqslant n$;

\emph{(2)} For any $\xi$-exact complex $\xymatrix{K_n\ar[r]&G_{n-1}\ar[r]&\cdots\ar[r]&G_1\ar[r]&G_0\ar[r]&A.}$ If all $G_i$ are $\xi$-$\mathcal{G}$projective, then so is $K_n$.

\end{prop}
 \begin{proof} $(1)\Rightarrow (2)$ By assumption, there exists an $\mathbb{E}$-triangle $\xymatrix{K\ar[r]^x&G\ar[r]^y&A\ar@{-->}[r]^\delta&}$ in $\xi$ where $G\in\mathcal{GP}(\xi)$ and $\xi$-$\mathcal{G}{\rm pd} K\leqslant n-1$. Since $\mathcal{C}$ { has} enough $\xi$-projectives, there exists an $\mathbb{E}$-triangle $\xymatrix{L\ar[r]^{x'}&P\ar[r]^{y'}&A\ar@{-->}[r]^{\delta'}&}$ in $\xi$ with $P\in\mathcal{P}(\xi)$. Then there exists a morphism $g$ such that $y'=yg$. { It follows} from Lemma \ref{lemb}(1) that there is a morphism $f$ which gives the following commutative diagram
$$\xymatrix{L\ar[r]^{x'}\ar@{-->}[d]_f&P\ar[d]_g\ar[r]^{y'}&A\ar@{=}[d]\ar@{-->}[r]^{\delta'}&\\
K\ar[r]^x&G\ar[r]^y&A\ar@{-->}[r]^\delta&}$$
such that $\xymatrix{L\ar[r]^{\tiny\begin{bmatrix}-f\\x'\end{bmatrix}\ \ \ }&K\oplus P\ar[r]^{\tiny\ \ \begin{bmatrix}x&g\end{bmatrix}}&G\ar@{-->}[r]^{y^*\delta'}&}$ is an $\mathbb{E}$-triangle. Hence it is an $\mathbb{E}$-triangle in $\xi$ as $\xi$ is closed under base change. Since $G$ and $P$ are $\xi$-$\mathcal{G}$projective, we can conclude that $\xi$-$\mathcal{G}{\rm pd} L=\xi$-$\mathcal{G}{\rm pd} K\leqslant n-1$ by Lemma \ref{pro4}. Consider the $\mathbb{E}$-triangle $\xymatrix{K_1\ar[r]&G_0\ar[r]&A\ar@{-->}[r]&}$ in $\xi$, similarly, we can prove that $\xi$-$\mathcal{G}{\rm pd} K_1=\xi$-$\mathcal{G}{\rm pd} L\leqslant n-1$.   The proof can be completed by induction.

$(2)\Rightarrow (1)$ Since $\mathcal{C}$ has enough $\xi$-projectives, there exists a $\xi$-exact complex
$$\xymatrix{K_n\ar[r]&P_{n-1}\ar[r]&\cdots\ar[r]&P_1\ar[r]&P_0\ar[r]&A}$$ with  all $P_i\in\mathcal{P}(\xi)$. Hence $K_n$ is $\xi$-$\mathcal{G}$projective by hypothesis, which implies  $\xi$-$\mathcal{G}{\rm pd} A\leqslant n$.
\end{proof}

\begin{lem}\label{lem6}   If $\xymatrix{A\ar[r]&B\ar[r]&C\ar@{-->}[r]&}$ is an $\mathbb{E}$-triangle in $\xi$ with $C\in\mathcal{GP}(\xi)$, then it is $\mathcal{C}(-,\widetilde{\mathcal{P}}(\xi))$-exact.
\end{lem}
\begin{proof} Since $C\in\mathcal{GP}(\xi)$, then there is an $\mathbb{E}$-triangle $\xymatrix{K\ar[r]^g&P\ar[r]^f&C\ar@{-->}[r]^\delta&}$ in $\xi$ with $P\in\mathcal{P}(\xi)$ and $K\in\mathcal{GP}(\xi)$. Next we claim that the $\mathbb{E}$-triangle $\xymatrix{K\ar[r]^g&P\ar[r]^f&C\ar@{-->}[r]^\delta&}$ is $\mathcal{C}(-,\widetilde{\mathcal{P}}(\xi))$-exact, that is, the sequence of abelian groups $$\xymatrix{(*)&0\ar[r]&\mathcal{C}(C, L)\ar[r]^{\mathcal{C}(f, L)}&\mathcal{C}(P, L)\ar[r]^{\mathcal{C}(g, L)}&\mathcal{C}(K, L)\ar[r]&0}$$ is exact for any $L\in\mathcal{C}$ with $\xi$-${\rm pd}L=n$. If $n=0$, then the sequence $(*)$ is exact by Lemma \ref{lem5}(2). Assume the sequence $(*)$ is exact for any $L\in\mathcal{C}$ with $\xi$-${\rm pd}L=n-1$. Now consider the case of $\xi$-${\rm pd}L=n$. Suppose that $\xymatrix{M\ar[r]^x&Q\ar[r]^y&L\ar@{-->}[r]^\rho&}$ is an $\mathbb{E}$-triangle in $\xi$ with $Q\in\mathcal{P}(\xi)$ and $\xi$-${\rm pd}M=n-1$. Hence the sequence of abelian groups $\xymatrix{0\ar[r]&\mathcal{C}(C, M)\ar[r]^{\mathcal{C}(f, M)}&\mathcal{C}(P, M)\ar[r]^{\mathcal{C}(g, M)}&\mathcal{C}(K, M)\ar[r]&0}$ is exact by induction. Since  the functor $\mathcal{C}(-, -)$ is biadditive functor, we have following commutative diagram.
$$\xymatrix{&0\ar[d]&0\ar[d]&\\
0\ar@{-->}[r]&\mathcal{C}(C, M)\ar[r]^{\mathcal{C}(C, x)}\ar[d]_{\mathcal{C}(f, M)}&\mathcal{C}(C, Q)\ar[d]^{\mathcal{C}(f, Q)}\ar[r]^{\mathcal{C}(C, y)}&\mathcal{C}(C, L)\ar[d]^{\mathcal{C}(f, L)}\ar@{-->}[r]&0\\
0\ar[r]&\mathcal{C}(P, M)\ar[r]^{\mathcal{C}(P, x)}\ar[d]_{\mathcal{C}(g, M)}&\mathcal{C}(P, Q)\ar[d]^{\mathcal{C}(g, Q)}\ar[r]^{\mathcal{C}(P, y)}&\mathcal{C}(P, L)\ar[d]^{ \mathcal{C}(g, L)}\ar[r]&0\\
0\ar@{-->}[r]&\mathcal{C}(K, M)\ar[r]^{\mathcal{C}(K, x)}\ar[d]&\mathcal{C}(K,Q)\ar[r]^{\mathcal{C}(K, y)}\ar[d]&{ \mathcal{C}(K,L)}\ar@{-->}[r]&0\\
&0&0&}$$
It is easy to see that the morphism $\mathcal{C}(C, x)$ is monic. For any morphism $h: C\rightarrow L$, there exists a morphism $h_1: P\rightarrow Q$ such that $hf=yh_1$. It follows from (ET3)$^{\rm op}$ that there exists a morphism $h_2$ which gives the following commutative diagram.
 $$\xymatrix{K\ar[r]^g\ar@{-->}[d]^{h_2}&P\ar[r]^f\ar[d]^{h_1}&C\ar@{-->}[r]^\delta\ar[d]^h&\\
 M\ar[r]^x&Q\ar[r]^y&L\ar@{-->}[r]^{\rho}&}$$
 Since $\mathcal{C}(g, M): \mathcal{C}(P, M)\rightarrow \mathcal{C}(K, M)$ is epic,  $h_2$ factors through $g$. Hence $h$ factors through $y$ by \cite[Corollary 3.5]{NP}. That is, $\mathcal{C}(C, y)$ is epic, which implies
that the sequence of abelian groups $\xymatrix{0\ar[r]&\mathcal{C}(C, M)\ar[r]^{\mathcal{C}(C, x)}&\mathcal{C}(C, Q)\ar[r]^{\mathcal{C}(C, y)}&\mathcal{C}(C, L)\ar[r]&0}$ is exact. Similarly one can prove that the sequence of abelian groups $\xymatrix{0\ar[r]&\mathcal{C}(K, M)\ar[r]^{\mathcal{C}(K, x)}&\mathcal{C}(K, Q)\ar[r]^{\mathcal{C}(K, y)}&\mathcal{C}(K, L)\ar[r]&0}$ is exact as $K\in\mathcal{GP}(\xi)$. It is straightforward to prove that the
 sequence $(*)$ is exact by  Lemma $3\times 3$ { in \cite[Corollary 3.6]{B"u}}. It follows from { Lemma \ref{lem5}(1)} that the $\mathbb{E}$-triangle $\xymatrix{A\ar[r]&B\ar[r]&C\ar@{-->}[r]&}$ is $\mathcal{C}(-,\widetilde{\mathcal{P}}(\xi))$-exact because any $\mathbb{E}$-triangle in $\xi$ which { ends} in $P$ is split.
\end{proof}

\begin{prop}\label{pro6} There is an inequality $\xi$-$\mathcal{G}{\rm pd} A\leqslant \xi$-${\rm pd} A$, and the equality holds if $\xi$-${\rm pd} A<\infty$.
\end{prop}
\begin{proof}   It is clear that $\xi$-$\mathcal{G}{\rm pd} A\leqslant \xi$-${\rm pd} A$. Set $\xi$-$\mathcal{G}{\rm pd} A=n$
and $\xi$-${\rm pd} A=m<\infty$, then $n\leqslant m$. Suppose $n<m$, there exists a $\xi$-projective resolution
$\xymatrix@C=1.5em{\cdots\ar[r]& P_1\ar[r]&P_0\ar[r] &A}$ of $A$  with all $P_i\in \mathcal{P}(\xi)$. For each integer $i$,
there exists an $\mathbb{E}$-triangle $\xymatrix@C=1.5em{K_{i+1}\ar[r]& P_i\ar[r]& K_i\ar@{-->}[r]&}$  in $\xi$. It follows from Proposition \ref{pro5} that
$K_n\in\mathcal{GP}(\xi)$,  and $K_{n+1}\in\widetilde{\mathcal{P}}(\xi)$.  Hence the $\mathbb{E}$-triangle
$\xymatrix@C=1.5em{K_{n+1}\ar[r]& P_n\ar[r]& K_n\ar@{-->}[r]&}$  is split since it is $\mathcal{C}(-,\widetilde{\mathcal{P}}(\xi))$-exact by Lemma \ref{lem6}, which implies  $K_{n}\in\mathcal{P}(\xi)$,
a contradiction. So $\xi$-$\mathcal{G}{\rm pd} A= \xi$-${\rm pd} A$.
\end{proof}

\begin{prop}\label{thm3} { Let $A$ be a nonzero object in $(\mathcal{C}, \mathbb{E}, \mathfrak{s})$. If $\xi$-$\mathcal{G}{\rm pd} A=n<\infty$}, then there exist $\mathbb{E}$-triangles $\xymatrix@C=1.5em{K\ar[r]&G\ar[r]&A\ar@{-->}[r]&
}$ and
$\xymatrix@C=1.5em{A\ar[r]&L\ar[r]&G'\ar@{-->}[r]&}$ in $\xi$ such that $G,G'\in\mathcal{GP}(\xi)$,
$\xi$-${\rm pd} K\leqslant n-1\ ($if $n=0$, this should be interpreted as $K=0)$ and $\xi$-${\rm pd} L\leqslant n$.
\end{prop}
\begin{proof}  If $n=0$, that is  $A\in\mathcal{GP}(\xi)$. Hence there exists an $\mathbb{E}$-triangle $\xymatrix@C=1.5em{0\ar[r]&A\ar[r]&A\ar@{-->}[r]^0&}$ in
$\xi$ such that $A\in\mathcal{GP}(\xi)$ and $K=0$. On the other hand, there exists an $\mathbb{E}$-triangle $A\to L\to G'\dashrightarrow$ in $\xi$
with $L\in\mathcal{P}(\xi)$ and $G'\in\mathcal{GP}(\xi)$.

If $n=1$, then there exists an $\mathbb{E}$-triangle $\xymatrix@C=1.5em{G_1\ar[r]&G_0\ar[r]&A\ar@{-->}[r]&}$  in $\xi$ with $G_0,G_1\in\mathcal{GP}(\xi)$.
Consider the $\mathbb{E}$-triangle $\xymatrix@C=1.5em{G_1\ar[r]&K\ar[r]&H\ar@{-->}[r]&}$ in $\xi$ with $K\in\mathcal{P}(\xi)$ and $H\in\mathcal{GP}(\xi)$.
Hence we have the following commutative diagram
$$\xymatrix{
  G_1 \ar[d] \ar[r] & G_0 \ar[d] \ar[r] & A \ar@{=}[d]\ar@{-->}[r]& \\
  K \ar[d] \ar[r] & G \ar[d] \ar[r] & A \ar@{-->}[r]& \\
  H \ar@{-->}[d]\ar@{=}[r] & H\ar@{-->}[d] & &  \\&&& }
$$
where all rows and columns are $\mathbb{E}$-triangles in $\xi$. It follows from Theorem \ref{thm1} that $G\in\mathcal{GP}(\xi)$,
and the middle row in the above diagram is the first required $\mathbb{E}$-triangle. Since $G\in\mathcal{GP}(\xi)$, there exists an $\mathbb{E}$-triangle $\xymatrix@C=1.5em{G\ar[r]&P\ar[r]&G'\ar@{-->}[r]&}$ in $\xi$ with $P\in\mathcal{P}(\xi)$ and $G'\in\mathcal{GP}(\xi)$.
Hence there exists the  following commutative diagram
$$\xymatrix{
  K \ar@{=}[d] \ar[r] & G \ar[d] \ar[r] & A \ar[d]\ar@{-->}[r]& \\
  K  \ar[r] & P \ar[d] \ar[r] & L \ar[d]\ar@{-->}[r]& \\
  &G'\ar@{-->}[d]\ar@{=}[r] & G' \ar@{-->}[d]&\\&&&    }
$$
{ where} all rows and columns are $\mathbb{E}$-triangles in $\xi$. It is clear that $\xi$-${\rm pd} L\leqslant 1$, and the third column in the above diagram
is the second required $\mathbb{E}$-triangle.

Assume that the results hold for { $n-1$, where $n\geqslant 2$}. Since $\xi$-$\mathcal{G}{\rm pd} A=n<\infty$, we have an $\mathbb{E}$-triangle
$\xymatrix@C=1.5em{K'\ar[r]&G_0\ar[r]&A\ar@{-->}[r]&}$  in $\xi$ with $G_0\in\mathcal{GP}(\xi)$ and  $\xi$-$\mathcal{G}{\rm pd}K'= n-1$.
Then there exists an $\mathbb{E}$-triangle $\xymatrix@C=1.5em{K'\ar[r]&K\ar[r]&G_1\ar@{-->}[r]&}$  in $\xi$ with $G_1\in\mathcal{GP}(\xi)$ and
$\xi$-${\rm pd}K\leqslant n-1$ by induction hypothesis on $K'$. Hence one can get the following commutative diagram
$$\xymatrix{
  K' \ar[d] \ar[r] & G_0 \ar[d] \ar[r] & A \ar@{=}[d]\ar@{-->}[r]& \\
  K \ar[d] \ar[r] & G \ar[d] \ar[r] & A \ar@{-->}[r]& \\
  G_1\ar@{-->}[d] \ar@{=}[r] & G_1\ar@{-->}[d] &&\\
  &&&    }
$$
{ where} all rows and columns are $\mathbb{E}$-triangles in $\xi$. It follows from Theorem \ref{thm1}  that $G\in\mathcal{GP}(\xi)$,
and the middle row in this diagram is the first required $\mathbb{E}$-triangle.
Since $G\in\mathcal{GP}(\xi)$, there is an $\mathbb{E}$-triangle $\xymatrix@C=1.5em{G\ar[r]&P\ar[r]&G'\ar@{-->}[r]&}$ in $\xi$ with
$P\in\mathcal{P}(\xi)$ and $G'\in\mathcal{GP}(\xi)$. Then one can get  the following commutative diagram
$$\xymatrix{
K \ar@{=}[d] \ar[r] & G\ar[d] \ar[r] & A \ar[d]\ar@{-->}[r]& \\
K  \ar[r] & P \ar[d] \ar[r] & L \ar[d]\ar@{-->}[r]& \\
&G'\ar@{-->}[d]\ar@{=}[r] & G' \ar@{-->}[d]&\\
&&&    }
$$
where all rows and columns are conflations in $\xi$, and $\xi$-${\rm pd} L\leqslant\mathcal{E}$-${\rm pd}K+1=n$. Hence the third column in the above diagram
is the second required $\mathbb{E}$-triangle.
\end{proof}

\subsection{Admissible model structure}

\begin{definition} \emph{(see \cite[Definition 4.1]{NP})}\label{df:cotorsion pair}
{\rm Let $\mathcal{U}$, $\mathcal{V}$ $\subseteq$ $\mathcal{C}$ be a pair of full additive subcategories, closed
under isomorphisms and direct summands. The pair ($\mathcal{U}$, $\mathcal{V}$) is called a {\it cotorsion
pair} on $\mathcal{C}$ if it satisfies the following conditions:

(1) $\mathbb{E}(\mathcal{U}, \mathcal{V})=0$;

(2) For any $C \in{\mathcal{C}}$, there exists a conflation $V^{C}\rightarrow U^{C}\rightarrow C$ satisfying
$U^{C}\in{\mathcal{U}}$ and $V^{C}\in{\mathcal{V}}$;

(3) For any $C \in{\mathcal{C}}$ , there exists a conflation $C\rightarrow V_{C} \rightarrow U_{C}$ satisfying
$U_{C}\in{\mathcal{U}}$ and $V_{C}\in{\mathcal{V}}$.}
\end{definition}

\begin{definition} \emph{(see \cite[Definition 4.2]{NP})}\label{df:cotorsion pair}
{\rm Let $\mathcal{X}$, $\mathcal{Y}$ $\subseteq$ $\mathcal{C}$ be any pair of full subcategories closed under
isomorphisms. Define full subcategories $\textrm{Cone}(\mathcal{X},\mathcal{Y})$ and $\textrm{CoCone}(\mathcal{X},\mathcal{Y})$ of
$\mathcal{C}$ as follows. These are closed under isomorphisms.

(1) $C$ belongs to $\textrm{Cone}(\mathcal{X},\mathcal{Y})$ if and only if it admits a conflation
$X \rightarrow Y \rightarrow C$ satisfying $X\in{\mathcal{X}}$ and $Y\in{\mathcal{Y}}$;

(2) $C$ belongs to $\textrm{CoCone}(\mathcal{X},\mathcal{Y})$  if and only if it admits a conflation
$C\rightarrow X \rightarrow Y$ satisfying $X\in{\mathcal{X}}$ and $Y\in{\mathcal{Y}}$.}
\end{definition}

\begin{definition} \emph{(see \cite[Definition 5.1]{NP})}\label{df:hovey-cotorsion pair}
{\rm Let ($\mathcal{S}$, $\mathcal{T}$) and ($\mathcal{U}$, $\mathcal{V}$) be cotorsion pairs on $\mathcal{C}$. Then $\mathcal{P}=((\mathcal{S}$, $\mathcal{T}$), ($\mathcal{U}$, $\mathcal{V}$)) is called a {\it twin cotorsion pair} if it satisfies
$\mathbb{E}(\mathcal{S}, \mathcal{V})=0$. Moreover, $\mathcal{P}$ is called a {\it Hovey twin cotorsion pair} if it satisfies
$\textrm{Cone}(\mathcal{V},\mathcal{S})$ = $\textrm{CoCone}(\mathcal{V},\mathcal{S})$.}
\end{definition}

In { \cite[Section 5]{NP}} Nakaoka and Palu laid out a correspondence between (nice enough) admissible model structures on extriangulated categories $\mathcal{C}$ and Hovey twin cotorsion pairs on $\mathcal{C}$. Essentially, an  admissible model structure on $\mathcal{C}$ is a Hovey twin cotorsion pair $\mathcal{P}=((\mathcal{S}$, $\mathcal{T}$), ($\mathcal{U}$, $\mathcal{V}$)) on $\mathcal{C}$. And an  {\it admissible model structure} on $\mathcal{C}$ is determined by the above cotorsion pairs in the following way:

(1) $f$ is a cofibration if it is an inflation with $\textrm{Cone}(f)\in{\mathcal{U}}$;

(2) $f$ is an acyclic cofibration if it is an inflation with $\textrm{Cone}(f)\in{\mathcal{S}}$;

(3) $f$ is a fibration if it is a deflation with $\textrm{CoCone}(f)\in{\mathcal{T}}$;

(4) $f$ is an acyclic fibration if it is a deflation with $\textrm{CoCone}(f)\in{\mathcal{V}}$;

(5) $f$ is a weak equivalence if there exist an acyclic fibration $g$ and an acyclic cofibration  $h$ such that $f=gh$.

The above correspondence makes it clear that an admissible model structure can be
succinctly represented by a Hovey twin cotorsion pair $\mathcal{P}=((\mathcal{S}$, $\mathcal{T}$), ($\mathcal{U}$, $\mathcal{V}$)). By a slight abuse of language we often refer to such a triple as an admissible model structure. This result is inspired
from \cite{Gillespie,HCc,HC}, where the case of exact categories is studied in more
details.

Let  $(\mathcal{C}, \mathbb{E}, \mathfrak{s})$ be an extriangulated category and $\xi$ a proper class, then $(\mathcal{C}, \mathbb{E}_\xi, \mathfrak{s}_\xi)$ is an extriangulated category by Theorem \ref{thma} where  $\mathbb{E}_\xi:=\mathbb{E}|_\xi$ and $\mathfrak{s}_\xi:=\mathfrak{s}|_{\mathbb{E}_\xi}$. As consequences of Propositions \ref{pro6} and \ref{thm3}, we have the following theorem.

\begin{thm}\label{thm:5.4} { Let $(\mathcal{C}, \mathbb{E}_\xi, \mathfrak{s}_\xi)$ be an extriangulated category as above.} If $n$ is a non-negative integer, then the following conditions are equivalent:

\emph{(1)} ${\rm sup}\{\xi$-$\mathcal{G}{\rm pd}A|A\in\mathcal{C}\}\leqslant n$.

\emph{(2)} $\mathcal{P}=((\mathcal{P}(\xi),\mathcal{C}), (\mathcal{GP}(\xi), \mathcal{P}^{\leqslant n}(\xi)))$ is an admissible model structure on $(\mathcal{C}, \mathbb{E}_\xi, \mathfrak{s}_\xi)$, where $\mathcal{P}^{\leqslant n}(\xi)=\{A\in\mathcal{C}|\xi$-${\rm pd}A\leqslant n\}$.
\end{thm}

\begin{proof}$(1)\Rightarrow(2)$. Assume that ${\rm sup}\{\xi$-$\mathcal{G}{\rm pd}A|A\in\mathcal{C}\}\leqslant n$. It is clear that $(\mathcal{P}(\xi),\mathcal{C})$ is a cotorsion pair in $(\mathcal{C}, \mathbb{E}_\xi, \mathfrak{s}_\xi)$ and $\mathbb{E}_{\xi}(\mathcal{GP}(\xi),\mathcal{P}^{\leqslant n}(\xi))=0$ by Lemma \ref{lem6}.  We have that $(\mathcal{GP}(\xi), \mathcal{P}^{\leqslant n}(\xi))$  is a cotorsion pair in $(\mathcal{C}, \mathbb{E}_\xi, \mathfrak{s}_\xi)$ by Proposition \ref{thm3}. To prove (2), we only need to check that $\textrm{Cone}(\mathcal{P}^{\leqslant n}(\xi),\mathcal{P}(\xi))$ = $\textrm{CoCone}(\mathcal{P}^{\leqslant n}(\xi),\mathcal{P}(\xi))$ in $(\mathcal{C}, \mathbb{E}_\xi, \mathfrak{s}_\xi)$.

Let $M$ be an object in $\textrm{Cone}(\mathcal{P}^{\leqslant n}(\xi),\mathcal{P}(\xi))$. Then there exists a  conflation
$X \rightarrow Y \rightarrow M$ in $(\mathcal{C}, \mathbb{E}_\xi, \mathfrak{s}_\xi)$ satisfying $X\in{\mathcal{P}^{\leqslant n}(\xi)}$ and $Y\in{\mathcal{P}(\xi)}$. It follows that $M$ $\in{\mathcal{P}^{\leqslant {n+1}}(\xi)}$. Thus $M$ $\in{\mathcal{P}^{\leqslant {n}}(\xi)}$ by Proposition \ref{pro6}, and hence there exists a conflation $M \rightarrow L \rightarrow G$  in $(\mathcal{C}, \mathbb{E}_\xi, \mathfrak{s}_\xi)$ such that $L\in{\mathcal{P}^{\leqslant n}(\xi)}$ and $G\in{\mathcal{GP}(\xi)}$ by Proposition \ref{thm3}. { One can check} that $G$ has finite $\xi$-projective dimension. It follows from Proposition \ref{pro6} that $G$ is $\xi$-projective. So $M$ belongs to $\textrm{CoCone}(\mathcal{P}^{\leqslant n}(\xi),\mathcal{P}(\xi))$, as desired. For the reverse containment, assume that $M$ is an object in $\textrm{CoCone}(\mathcal{P}^{\leqslant n}(\xi),\mathcal{P}(\xi))$. Then there exists a  conflation
$M\rightarrow X \rightarrow Y$ in $(\mathcal{C}, \mathbb{E}_\xi, \mathfrak{s}_\xi)$ satisfying $X\in{\mathcal{P}^{\leqslant n}(\xi)}$ and $Y\in{\mathcal{P}(\xi)}$. It is clear that the conflation $M\rightarrow X \rightarrow Y$  is split. Thus $M$ is in $\mathcal{P}^{\leqslant n}(\xi)$, and hence $M$ belongs to $\textrm{Cone}(\mathcal{P}^{\leqslant n}(\xi),\mathcal{P}(\xi))$.

$(2)\Rightarrow(1)$ is straightforward by noting that $(\mathcal{GP}(\xi), \mathcal{P}^{\leqslant n}(\xi))$  is a cotorsion pair in $(\mathcal{C}, \mathbb{E}_\xi, \mathfrak{s}_\xi)$.
\end{proof}

As a consequence of Theorem \ref{thm:5.4}, we have the following corollary which { generalizes} \cite[Theorem 5.7]{Yang}.

\begin{cor}\label{yang} Let $\mathcal{C}$ be a triangulated category and $\xi$ a  proper class of triangles. If $n$ is a non-negative integer, then the following conditions are equivalent:

\emph{(1)} ${\rm sup}\{\xi$-$\mathcal{G}{\rm pd}A|A\in\mathcal{C}\}\leqslant n$.

\emph{(2)} $\mathcal{P}=((\mathcal{P}(\xi),\mathcal{C}), (\mathcal{GP}(\xi), \mathcal{P}^{\leqslant n}(\xi)))$ is an admissible model { structure on} $(\mathcal{C}, \mathbb{E}_\xi, \mathfrak{s}_\xi)$.
\end{cor}

\section*{\bf Acknowledgments}

This work was carried out while the corresponding
author was visiting at Universit${\rm\acute{e}}$ de Sherbrooke with a support by China Scholarship Council. He thanks Prof. Shiping Liu and the faculty of D${\rm\acute{e}}$partement de Math${\rm\acute{e}}$matiques for their hospitality. The authors would like to thank the referee for reading the paper carefully and for many suggestions on mathematics and English expressions.

\renewcommand\refname
\end{document}